\documentclass[11pt,leqno]{article}

 \usepackage[all]{xy} 
\usepackage{amssymb,amsmath,amsthm,  url,rotating}
\usepackage{setspace}
 \usepackage{graphicx,epsfig}
 \usepackage{xcolor,graphicx}
  \usepackage{makeidx}
  \usepackage[all]{xy}
\usepackage{amssymb,amsmath,amsthm,  url,rotating}
\usepackage{setspace}
\usepackage{mathrsfs}
 \usepackage{graphicx,epsfig}
 \usepackage{xcolor,graphicx}
  \usepackage{makeidx}
\newcommand{\vp}{\varepsilon}
\newcommand{\bb}[1]{\mathbb{#1}}
\newcommand{\cl}[1]{\mathcal{#1}}

\newcommand{\ovl}{\overline}

\theoremstyle{plain}
\newtheorem{thm}{Theorem}[section]
\newtheorem{lem}[thm]{Lemma}

\newtheorem{pro}[thm]{Proposition}

\newtheorem{cor}[thm]{Corollary}

\theoremstyle{definition}

\newtheorem{dfn}[thm]{Definition}

\theoremstyle{remark}
\newtheorem{rem}[thm]{Remark}

\numberwithin{equation}{section}

\def\tilde{\widetilde}

\renewcommand{\tilde}{\widetilde}

\def\C{\cl C}
\def\CC{\bb C}

\def\F{\bb F}

\def\d{\delta}
\def\NN{\bb N}

\def\CC{\bb C}

\def\CC{\bb C}

\def\F{\bb F}

\def\d{\delta}
\def\NN{\bb N}

\def\CC{\bb C}

\def\phi{\varphi}

\def\nl{\nolimits}

\def\a{\alpha}
    
     \def\b{\beta}
     \def\D{\Delta}
\makeindex

\def\C{\mathscr{C}}
\def\B{\mathscr{B}}
\def\I{\cl  I}

  \def\a{\alpha}

\begin{document}

 \pagenumbering{roman}

 \thispagestyle{empty}

\setcounter{page}{1}

\newtheorem{ass}{}

 \pagenumbering{arabic}
 \def\ovl{\overline}
 \def\U{\cl U}

      \def\l{\ell}
      
       \title{Operator space theory from the maximal viewpoint}

\author{by\\
Gilles  Pisier\footnote{ORCID    0000-0002-3091-2049} 
\\ Sorbonne Universit\'e and Texas  A\&M  University\\
and \\
Jean Roydor  \\
Sorbonne Universit\'e}

 \maketitle
     
      \begin{abstract}  We initiate a  theory of ``operator spaces"
      (in other words subspaces of $C^*$-algebras) with an emphasis on 
      the maximal tensor product. The natural morphisms become the maximally bounded maps instead of the completely bounded ones.
    With this new looking glass  we first examine various issues : lifting and extension properties, local reflexivity, inner derivations, similarity problems.
      This first part is semi-expository while the second one 
      focuses on the class of von Neumann algebras $M$ such that any
     completely bounded   unital homomorphism on a $C^*$-algebra with values in $M$ is maximally bounded (or equivalently is   similar to a $*$-homomorphism via a similarity belonging to $M$). 
     We show that this holds if and only if $M$ is the direct sum
     of a finite von Neumann algebra and an injective one. 
     This also characterizes the $M$'s such that all completely bounded derivations into $M$ are inner.
     Our main tool for this is  a new characterization
     of injective von Neumann algebras. \footnote{AI usage statement: AI tools were not used to develop the ideas or check proofs in this paper.}

  \end{abstract}

Operator spaces   can be defined as subspaces of $C^*$-algebras.
Any subspace $E\subset A$ of a $C^*$-algebra $A$ inherits from $M_n(A)$ a   norm on $M_n(E)$.
The sequence of norms $(\|\ \|_{M_n(E)})$ is what is called the operator space 
structure on $E$. For any vector space $E$, Ruan's theorem (see \cite{ER}) characterizes the ``abstract" sequence of norms on $M_n(E)$ that come from a ``concrete" realization of $E$ as a subspace of a $C^*$-algebra.
This sequence of norms then determines the norm in the minimal (or spatial) tensor products of $E$ with any other
operator space.  

The theory of operator spaces (o.s. in short) as presented in \cite{ER,P4, Pau}
rests on   completely bounded (or c.b.) maps. These are the maps
between operator spaces  that preserve the minimal (or spatial) tensor products.
In  the semi-expository first part of this paper we show that-to some extent-a parallel theory can be developed
based on the maximally bounded (or m.b.) maps, i.e. those that  preserve the maximal tensor product.

Before giving a more detailed presentation, let us outline the contents of this paper.
Our  \S \ref{bg}   
is devoted to background on   $C^*$-algebraic tensor products. 
In \S \ref{mb} we present the basic properties of m.b. maps.
In \S \ref{dec} we   review the properties of the decomposable maps following mainly Haagerup's
\cite{Ha1132}.   Decomposable maps are those that are linear combinations of completely positive maps. They are closely related to m.b. maps.
 In parallel we  consider
maximally positive (m.p. in short) maps and  their relation to  m.b. maps. 
In \S \ref{LR} we compare the local reflexivity
of the min and max tensor products. In \S \ref{lift} we describe the known characterization of the lifting property (LP in short) in terms of maximal tensor product, and
we show (in sharp contrast with the c.b. case)  that the existence of
 local mb-embeddings into $C^*(\F_\infty)$ is equivalent to the LP.
  
  In \S \ref{ext} we make a parallel review for the extension properties.

In \S \ref{dual} we explore what the duality of operator spaces becomes in our new framework. 
In \S \ref{schur} we briefly survey what it means for Schur multipliers to be m.b.

In the next three sections we investigate the homomorphisms $u$  from a $C^*$-algebra $A$ into a von Neumann algebra $N\subset B(H)$
for which  there is an invertible $S\in N$
such that $S u(.)S^{-1}$ is a $*$-homomorphism.
We say that these are  ``strongly similar" to a $*$-homomorphism to distinguish them
from those for which this holds for an invertible in $B(H)$. The latter  are commonly termed  similar to a $*$-homomorphism. 
We show in \S \ref{s2} that strong similarity holds for $u$ if and only if $u$ is m.b.
and moreover $\|u\|_{mb}$ coincides with the minimum of $\|S\| \|S^{-1}\|$ 
over all invertibles $S\in N$
such that $S u(.)S^{-1}$ is a $*$-homomorphism (see Theorem \ref{s1}).
 We also
consider an analogous question for (inner) derivations (see Theorem \ref{t5.6}) and we give 
a version of  Arveson's distance formula for the m.b. norm of a derivation.
These results are analogues for m.b. maps of well known results for the c.b. case due to Haagerup \cite{Ha} for similarities,
to   Christensen \cite{Chr} for derivations. See also  \cite{95}.
In the case of a non self-adjoint unital
operator algebra $A\subset B(\cl H)$,  we prove the ``maximal analogue" of
Paulsen's criterion  for similarity to a completely contractive homomorphism (see Theorem \ref{pam}). 

The notion of strong similarity is considered in \cite[Th.1]{toho} when $N$ is  the von Neumann algebra
generated by $u(A)$ and it is proved there  that $A$ is nuclear if  and only if 
strong similarity holds for  any $u: A \to N$ for all $N$'s.
In the present paper we focus our interest on the class of $N$'s for which
strong similarity holds for all $A$'s.
In \S \ref{con} we construct concrete examples of c.b. homomorphisms
from a $C^*$-algebra $A$  to a von Neumann algebra $M$ that are not m.b.
 We may take $A=C^*(\F_\infty)$ or $A=C_\lambda^*(\F_\infty)$ and $M=B(\ell_2)\bar\otimes M_\lambda$
 where $M_\lambda$ is the von Neumann algebra of the free group $\F_\infty$.
In \S \ref{inj} we give a new characterization of injective von Neumann algebras.
In \cite{HP2} it is shown that if $M$ is such an algebra then there is a lower bound
of the form $\|u\|_{cb} \|v\|_{cb} \ge c \sqrt n$ ($c>0$) for any factorization
$ R_n\cap C_n {\buildrel u \over \rightarrow } M {\buildrel v \over \rightarrow } R_n\cap C_n$
and any $n$.  
We prove in Theorem \ref{t3.1} that conversely this property implies that $M$ is injective.

In \S \ref{svss} we show that a von Neumann algebra $N$ is such that 
all c.b. unital homomorphisms into $N$ are strongly similar to $*$-homomorphisms
(i.e. are m.b.) if and only if $N$ is the direct sum of a finite von Neumann algebra and an injective one.

As   detailed in \S \ref{bg} the minimal tensor product is ``injective"
while the maximal one is in some sense ``projective"  or rather ``exact", as the latter term would be more appropriate.
 
We abbreviate completely bounded by c.b.
and maximally bounded by m.b. In addition operator space is abbreviated
by o.s.  finite dimensional by f.d.  completely positive by c.p. and completely isometric by c.i.
The notation $ CB(E,F)$ (resp. $CP(E,F)$) stands for the set of c.b. (resp. c.p.) maps from $E$ to $F$.

Let $E\subset A$ and $F\subset B$ be o.s. realized as subspaces of $C^*$-algebras $A,B$. Recall   that the cb-distance between  $E,F$ is defined by
     $$d_{cb}(E,F)=\inf\{ \|u\|_{cb}\|u^{-1}\|_{cb}  \} $$
     where the inf runs over all isomorphisms $u:E\to F$.     Of course   we may define in exactly the same manner
     $$d_{mb}(E,F)=\inf\{ \|u\|_{mb}\|u^{-1}\|_{mb}  \} ,$$
 for  the mb-norm as defined in   \eqref{e88} or \eqref{e100} below in 
 parallel with the cb-analogues \eqref{e7} or \eqref{e8}.
 
 We recall the basic facts on the min and max tensor products
of $C^*$-algebras in \S \ref{bg}. A major role will be played by the full
$C^*$-algebra $C^*(\F_\infty)$ of the free group $\F_\infty$ with infinitely (countably)
many generators. For convenience we denote it by
$$\C=C^*(\F_\infty).$$

\begin{rem}\label{fund}
The (non-separable) algebra $B(\ell_2)$  plays a fundamental role in the usual operator space theory
because any separable $C^*$-algebra embeds in it as a $C^*$-subalgebra.
The analogous fundamental fact for our ``maximal" theory is that any separable unital $C^*$-algebra is a quotient of $\C$. 
\end{rem}
 
\section{Background on $C^*$-norms}\label{bg}
   
   Tensor products of $C^*$-algebras were studied by numerous authors.
 They originated in Japan in the late 1950's in
 Turumaru's work, then
 Takesaki (1958) and Guichardet (1965) introduced the minimal and maximal
 $C^*$-norms, 
 followed by a series of major contributions among which
 Lance (1973)
 Choi-Effros (+Connes),  Kirchberg (1976-7), Wasserman (1976),
  Effros-Lance (1977),
  Archbold-Batty (1980),
   Effros-Haagerup (1985),
   Kirchberg (1993) and others after him
   $\cdots$.

In analogy with Grothendieck's thesis for tensor products of  Banach spaces
  Takesaki and  Guichardet discovered  that there   was a minimal and a maximal tensor product
 denoted respectively by $A\otimes_{\min} B$
 and $A\otimes_{\max} B$.
 
 These are obtained by completing the algebraic $A \otimes B$
 with respect to the  minimal and maximal $C^*$-norms
 $\|\ \|_{\min}$ or $\|\ \|_{\max}$.
 
When $A\subset B(H)$ and $B \subset B(K)$, 
any $t\in A \otimes_{alg} B$ defines an operator on $H \otimes_2 K$,
so that we may define

$$  \|t\|_{\min}=\|t\|_{B(H \otimes_2 K)} \text{``spatial norm"}$$
and
$$  \|t\|_{\max}=\sup\{ \|\pi \cdot \sigma (t)\|_{B(\cl H)} \mid \pi,\sigma\  \text{with commuting ranges} \}$$
where the sup runs over all $\cl H$'s and all pairs $(\pi,\sigma)$ of  $*$-homomorphisms 
on the same Hilbert space
$\xymatrix{
A \ar[r]^{\pi} &B(\cl H) & B \ar[l]_{ \sigma}}$ with commuting ranges.

To place the minimal  norm in parallel with the maximal one
note that we also have
$$\|t\|_{\min}=
\sup\{ \|[\pi \otimes\sigma] (t)\|_{B(  H_\pi \otimes_2 H_\sigma)} \}$$
where the sup runs this time over all $*$-homomorphisms $\pi,\sigma$.
Since $  Id\otimes  \sigma$ and $\pi \otimes Id $ have commuting ranges, we have
$\|t\|_{\min} \le \|t\|_{\max}$.
These are $C^*$-norms. By a $C^*$-norm on $A \otimes_{alg} B$
we mean a Banach algebra norm such that $t\mapsto t^*$ is isometric
and $\|t^*t\|=\|t\|^2$ for all $t \in A \otimes_{alg} B$.
It is easy to see that the max-norm is the largest possible $C^*$-norm
and not so easy to see (this is Takesaki's theorem) that the min-norm
is the smallest possible one. Thus for any $C^*$-norm $\|\ \|$ on $A\otimes_{alg} B$  
  $$  \|\ \|_{\min}  \le \|\ \|\le \|\ \|_{\max} .$$

The basic properties are as follows, here $A,B,C, D$ are $C^*$-algebras:

First we may flip  ($a\otimes b\mapsto b\otimes a$)  the factors without changing the norm, in other words we have canonical isometric isomorphisms:
$B\otimes_{\min}A\simeq A\otimes_{\min}B$ and    $B\otimes_{\max}A\simeq A\otimes_{\max}B$ {  (``commutativity")}.

For any embedding $B\subset C$
$$    A\otimes_{\min} B \subset A\otimes_{\min} C \quad\text{  (``injectivity")}.$$

For any quotient $C/\cl I $ we have
$$A\otimes_{\max} [C/\cl I  ] =[A\otimes_{\max} C]/[A\otimes_{\max} \cl I  ]  \quad\text{  (``projectivity")}.$$
Flipping the factors we also have of course:
$$\forall B\subset C \qquad    B\otimes_{\min}A  \subset C\otimes_{\min}  A 
\text{  and  } 
 [C/\cl I  ] \otimes_{\max} A=[C\otimes_{\max} A]/[\cl I\otimes_{\max}   A]   .$$

Thus the minimal tensor product is injective while the maximal one is projective.
It is important to emphasize that, in sharp contrast, generally the minimal one
is   {\it not} projective and the maximal one {\it not} injective (see \cite[\S 7]{157}).
 
\begin{rem}\label{R8}
Let $u: E \to F$  
be a linear map between operator spaces  $E\subset A$ and $ F\subset   B$ 
with $A,B$ denoting  $C^*$-algebras as usual.
Given norms $\alpha,\beta$    on the algebraic tensor products $D \otimes_{alg} A$
and $D \otimes_{alg} B$,
assume that 
$\| Id_D \otimes u: (D \otimes_{alg} E, \|\ \|_{D \otimes_{\a}  A}   ) \to (D \otimes_{alg} F, \|\ \|_{D \otimes_{\b}  B}   )  \|  <\infty.$
Then $Id_D \otimes u$ extends by density to a bounded linear map 
  from $\ovl{D \otimes_{alg} E}^\a$ to $\ovl{D \otimes_{alg} F}^\b$.
  At the risk of abusing the notation we set
  $D \otimes_{\a} E=\ovl{D \otimes_{alg} E}^\a$,   
 $D \otimes_{\b} E=\ovl{D \otimes_{alg} F}^\a$ and
 we still denote by $ Id_D \otimes u : D \otimes_{\a} E\to D \otimes_{\b} F$ 
 the resulting map.
Then we may write
\begin{equation}\label{e8}\|u\|_{cb} = \sup\nl_{D} \| Id_D \otimes u: D \otimes_{\min} E\to D \otimes_{\min} F  \| ,\end{equation}
where the sup runs over all $C^*$-algebras $D$.
Moreover the sup is attained on the subcollection $\{M_n\mid n\ge 1\}$.
It is also attained on any large enough $D$ where large enough means that for any $n$
the identity of $M_n$ factors completely contractively through $D$. See e.g. \cite[p. 15]{157}.
In particular, this holds for
for $D=K(\ell_2)$   (the compact operators on $\ell_2$) or for $D=B(\ell_2)$, but also for  $D=\C$ which is    less standard,  whence
\begin{equation}\label{e7} \|u\|_{cb} = \| Id_\C \otimes u: \C \otimes_{\min} E\to \C \otimes_{\min} F  \| ,\end{equation}
which should be compared with our choice of notation 
\begin{equation}\label{e88} \|u\|_{mb} = \| Id_\C \otimes u: \C \otimes_{\max} E\to \C \otimes_{\max} F  \| ,\end{equation}
for which there is more to come.
\end{rem}

Lastly, we recall a classical result  going back to    Lance's early work on the maximal tensor product (see \cite[p.218]{Tak} or \cite[p. 229]{P4}):
\begin{pro}\label{imp}
Consider  $C^*$-algebras $ A_j$, $  B_j$ ($j=1,2$) and maps
$u_j\in CP(A_j,B_j)$. Then $u_1 \otimes u_2$ extends to
a bounded c.p. map from $A_1\otimes_{\max} B_1$ to $A_2\otimes_{\max} B_2$ and
$$\|u_1 \otimes u_2 : A_1\otimes_{\max} B_1 \to A_2\otimes_{\max} B_2\|\le \|u_1  \| \|u_2\|.$$
\end{pro}

 \section{Maximally bounded maps}\label{mb}

Given   closed subspaces $E_j\subset A_j$ of  $C^*$-algebras $A_j$ ($j=1,2$), we define the max-norm on the algebraic tensor product $E_1 \otimes_{alg} E_2$ as the norm induced on it by the
maximal tensor product of the $C^*$-algebras $A_1 \otimes_{\max} A_2$.
We denote by
$$E_1 \otimes_{\max} E_2 \subset A_1 \otimes_{\max} A_2$$
the completion of the resulting normed space, whence a new o.s.

One major difference with the min-norm of ordinary o.s. theory is that
the max-norm depends strongly on the ambient $C^*$-algebras $A_1,A_2$.
Actually this is already the case when $E_1,E_2$ are $C^*$-subalgebras: if
we are given embeddings $A_j\subset B_j$ into larger $C^*$-algebras,
the   natural morphism  $A_1 \otimes_{\max} A_2 \to B_1 \otimes_{\max} B_2$ is in general not injective (equivalently not isometric).

    Let $A,B$ be $C^*$-algebras, with operator spaces $E\subset A$ and $F\subset B$. Following  \cite{157}
    a linear map $u: E \to F$ is called maximally bounded (m.b. in short)
    if the map $Id_\C \otimes u: \C \otimes_{alg} E \to \C \otimes_{alg} F$
    is bounded when both its domain and range are equipped with
    the (just defined) maximal norm,
   The mb-norm is then defined  as follows
   $$\|u\|_{mb} = \| Id_\C \otimes u: (\C \otimes_{alg} E, \|\ \|_{\C \otimes_{\max}  A}   ) \to (\C \otimes_{alg} F, \|\ \|_{\C \otimes_{\max}  B}   )  \| ,$$
   and the space of such m.b. maps is denoted by $MB(E,F)$, in analogy with  
   $CB(E,F)$ in the completely bounded (c.b. in short) case. Even though this is somewhat non-standard, the latter can be defined
   in the same way but using the minimal $C^*$-norm in place of the maximal one (see Remark \ref{R8}).
   
   Assume $A\subset B(H)$ and $B\subset B(\cl H)$ as $C^*$-subsalgebras.
   Then if we consider $E$ and $F$ as subspaces of $B(H)$ and  $B(\cl H)$
   the max-norm on $\C \otimes E$ and $\C \otimes F$ coincides with the min
   norm;  in fact this is true whenever $B(H)$ and  $B(\cl H)$
   are replaced by  $C^*$-algebras with the WEP (this is due to Kirchberg, see \S \ref{ext}). It follows that in this case
   the mb-norm of  $u: E \to F$  coincides with its cb-norm. With this looking glass the theory of cb-maps and  o.s. is ``embedded" in that of mb-maps.

      One striking feature
 of the mb-norm that diverges from the cb-case is that if $B$ is an {\it arbitrary}
  $C^*$-algebra
 any linear map $u: E \to B$ with   $\|u\|_{mb}=1$ ``extends" to a map
 $\tilde u: A \to B^{**}$  with $\|\tilde u\|_{mb}=1$.
 Therefore if $B$ is an {\it arbitrary}
   von Neumann algebra any  $u: E \to B$ with   $\|u\|_{mb}=1$  extends to a map
 $\tilde u: A \to B$  with $\|\tilde u\|_{mb}=1$. See Th. \ref{d1}  , \S \ref{ext}  and \cite[\S 6.3]{154} for more on this theme.

The properties of the max-tensor product imply that for any $C^*$-algebra
    $D$ we have
     \begin{equation}  \label{e10} \| Id_D \otimes u: D \otimes_{max} E  \to D \otimes_{\max} F  \| \le   \|u\|_{mb} , \end{equation}
     and hence
      \begin{equation}  \label{e100} 
      \|u\|_{mb}=\sup\nl_D\| Id_D \otimes u: D \otimes_{max} E  \to D \otimes_{\max} F  \|  .\end{equation}
     
     Together with the second assertion in Remark \ref{fund}, \eqref{e10}   follows from the   facts that  firstly
  we may restrict to the case when $D$ is separable and unital (exercise) 
 and hence that we may assume $D=\C/\I$, and secondly  
  for any quotient $D=C/\I$ of a $C^*$-algebra $C$ we have a natural isometric identity
    $$D \otimes_{max} E= [C \otimes_{max} E]/[\I \otimes_{max} E] .$$

    Actually a moment of thought shows that
     \begin{equation}  \label{e10'} \| Id_D \otimes u: D \otimes_{max} E  \to D \otimes_{\max} F  \|_{mb} \le   \|u\|_{mb} . \end{equation}
    Furthermore, for any o.s. $\cl G\subset D$ we may state
      \begin{equation}  \label{e103} \| Id_\cl G \otimes u: \cl G \otimes_{max} E  \to \cl G \otimes_{\max} F  \|_{mb} \le   \|u\|_{mb} . \end{equation}
    \begin{rem}\label{comp}
We note for further reference that
given o.s.  $E\subset A$, $F\subset B$. $G\subset C$ 
and maps $u\in MB(E, F)$ and $v\in MB(F, G)$, we have
$vu\in MB(E, G)$ and 
\begin{equation}\label{e21}\|vu\|_{mb} \le \|v\|_{mb}\|u\|_{mb}.\end{equation}
\end{rem}

    In analogy with the well known c.b. case, using $u_1\otimes u_2= [u_1 \otimes Id][Id \otimes u_2 ] $ 
    and \eqref{e103}  we may state:
    \begin{pro} Consider o.s. $E_j\subset A_j$, $F_j\subset B_j$ ($j=1,2$).
    For any pair of  m.b. maps $u_j: E_j \to F_j$, the map $u_1\otimes u_2$
    extends to a map in $MB(E_1\otimes_{\max} E_2, F_1\otimes_{\max} F_2)$ such that
     \begin{equation}  \label{e10''}
     \|u_1\otimes u_2\|_{mb} \le \|u_1 \|_{mb}\| u_2\|_{mb}.
      \end{equation}
    \end{pro}

 \begin{rem}\label{R2}  Note that for any map $u$ defined on $A$
    and any quotient map $q: \C \to A  $ 
    we have $$\| u q\|_{mb}=\|u\|_{mb}.$$  
\end{rem}

\begin{rem}[``max-isometric"]\label{R1}
In analogy with complete isometries, we say that a map
$u: E \to F$ is  
 maximally isometric or 
max-isometric 
if it is m.b. and the associated map $Id_\C \otimes u: \C \otimes_{max} E
\to C \otimes_{max} F$ is isometric.  In the $C^*$-framework of \cite{154},
such maps are called max-injective. \\ Obviously 
the restriction of a max-isometric map to a further subspace is still
max-isometric.\\
For instance   for any o.s.  $E$ the inclusion $E\subset A$ is max-isometric (this holds by the very definition of $\C \otimes_{max} E$). More significantly,  this holds whenever $E$ is an ideal. In particular the inclusion $A\to \hat A$ of $A$ into its unitization is  max-isometric.
Moreover, 
   if $A$ is a $C^*$-algebra the inclusion $A \subset A^{**}$
is max-isometric. See   \cite[\S 7.3]{154} for details on all this. \\
For an o.s. $E\subset A$, we   view 
$E^{**}$ as an o.s. using the natural embedding 
$E^{**}\subset A^{**}$, then the inclusion $i_E: E \to E^{**}$ is max-isometric. 
 
Also for unital $C^*$-algebras $A,B$ 
the maps $u_1: A\to A \otimes_{max} B$ (resp. $u_2: B\to A \otimes_{max} B$) defined by $u_1(a)=a \otimes 1$ (resp. $u_2(b)=1 \otimes b$) are 
 max-isometric. Thus for o.s. $E\subset A$ and $F\subset B$
 the same is true for the restrictions of   $u_1,u_2$ to $E$ and $F$.
 \end{rem}

\begin{rem}[``maximal" analogue of an o.s. structure]
Consider a vector space $E$ given with two o.s. structures $E\subset A$ and $E\subset B$,  meaning two linear embeddings that represent it ``concretely".
Let $E_{\subset A}$ and $E_{\subset B}$ denote the corresponding structures.
The two structures are considered ``the same" if
$Id_E$ (the identity on $E$) defines a complete isometry
 from $E_{\subset A}$ to $E_{\subset B}$. 
 In that case the norms induced on $E\otimes_{alg} D$ by $E_{\subset A}\otimes_{\min} D$
 and $E_{\subset B}\otimes_{\min} D$ coincide for any $C^*$-algebra $D$,
 as well as the spaces $CB(E_{\subset A},F)$ and $CB(E_{\subset B},F)$, or
  $CB(F, E_{\subset A})$ and $CB(F, E_{\subset B})$ for any o.s. $F\subset C$.
In this framework, an ``abstract o.s."  is an equivalent class of o.s. up to
completely isometric isomorphism.
\\ In the new ``maximal viewpoint" that we are promoting in this paper, 
$E_{\subset A}$ and $E_{\subset B}$ will be considered ``the same"  if
$Id_E$ (the identity on $E$) defines a max-isometry
 from $E_{\subset A}$ to $E_{\subset B}$. 
  In that case the norms induced on $E\otimes_{alg} D$ by $E_{\subset A}\otimes_{\max} D$
 and $E_{\subset B}\otimes_{\max} D$ coincide for any $C^*$-algebra $D$,
 as well as the spaces $MB(E_{\subset A},F)$ and $MB(E_{\subset B},F)$, or
  $MB(F, E_{\subset A})$ and $MB(F, E_{\subset B})$ for any o.s. $F\subset C$.\\
  In other words, our maximal analogue of  abstract o.s.
  is   an equivalence class of  o.s. modulo maximally isometric isomorphism.
   \end{rem}

  \begin{rem}[Important warning] When $E\subset A$ and $F\subset B$ are $C^*$-{\it subalgebras} and $\pi: E \to F$ is a $*$-homomorphism, 
  clearly 
  $$\|\pi\|_{MB(E_{\subset E},F)}=1$$
  but it is not true in general that $\|\pi\|_{MB(E_{\subset A},F)}=1$, unless
  the inclusion $E\subset A$ is max-isometric.
  \end{rem}
  
  \section{Decomposable and maximally positive maps}\label{dec}
The notion of mb map is 
closely related to
 that of
decomposable map. A map $u$ that is a linear combination of c.p.  maps
  is called decomposable.
Clearly four suffice, i.e. we have
$u=u_1-u_2+i(u_3-u_4)$ with $u_j$ c.p.
This is analogous to the decomposition of complex measures.
For obvious reasons, such maps are useful only when
the domain has many positive elements, so
we will restrict to maps between operator systems $E\subset A$, $F\subset B$. We denote by $D(E,F)$ the space of such maps.
When  $u$ is selfadjoint (i.e. $u=u_*$ where by definition $u_*(x)=u(x^*)^*$)
and decomposable
we have $u=u_1-u_2$ with $u_j$ c.p. and the decomposable norm
is defined as $\|u\|_{dec} =\inf \|u_1+u_2\|$ where the inf runs over all possible
decompositions.
 But for a general $u: E \to F$ 
the suitable definition of $\|u\|_{dec}$ is more involved.
Following Haagerup \cite{Ha1132} we denote
 \begin{equation}\label{d11}\| u\|_{dec}=\inf\{\max\{\| S_1\|,\| S_2\|\}\}
 \end{equation}
  where the infimum runs over all maps $S_1 ,S_2\in CP(E,F)$ such that the map
 \begin{equation}\label{d12} V: x\to \left(
\begin{matrix}
 S_1 (x) & u(x)\\
 u(x^* )^* & S_2 (x)
\end{matrix}
\right) \end{equation}
is in $CP(E,M_2 (F))$.  
Then $D(E,F)$ equipped with the norm $\| \ \|_{dec}$ is a Banach space.

\begin{rem}\label{R6}
Note that 
if there are maps $S_1 ,S_2\in CP(E,F)$ with $\|S_j\|\le 1$ ($j=1,2$) such that
$V\in CP(E,M_2 (F))$, then there are
{\it unital}   maps $S_1 ,S_2\in CP(E,F)$
with the same property.
   Indeed,
let $f$ be a state (meaning a positive linear form on $E$ such that $f(1)=1$).
Note  $0\le S_j(1)\le 1$.
Let
$$V_1(x)= \left(
\begin{matrix}
 S_1 (x)+f(x)(1-S_1(1)) & u(x)\\
 u(x^* )^* & S_2 (x) +f(x)(1-S_2(1))
\end{matrix}
\right).
$$
Then $V_1- V\in CP$  hence $V_1\in CP$ and  its diagonal entries are now unital.
\end{rem}

 Incidentally it remains an open question (already raised by Haagerup in \cite{Ha1132}) whether
 the inf in the preceding definition is always attained (of course it is so,
 by a weak* compactness argument, when 
 the range is a dual space ).
 
 \begin{rem}\label{nor} Consider a linear map $u:E \to F$ between two operator systems
 $E\subset A$, $F\subset B$. If $u$ is positive then it is bounded
 (even if $E$ is non closed) with norm at most $2\|u(1)\|$.
 If $u$ is c.p. then $\|u\|=\|u(1)\|$.\\
  If $u(1)=1$ then $u$ is positive  if     $\|u\| =1$.\\
 If $u(1)=1$ then $u$ is c.p. if  and only if  $\|u\|_{cb}=1$.\\
 See \cite[chap. 2]{Pau} for detailed proofs.
\end{rem}

 Following  \cite{157}, a linear map $u:E \to F$ between two operator systems
 $E\subset A$, $F\subset B$ is called maximally positive (m.p. in short)
 if  for all $C^*$-algebras $D$ the map $Id\otimes u: D \otimes_{alg} E \to D \otimes_{alg} F$ is such that
  $$[Id_D \otimes
 u]  ( ( D \otimes_{alg} E)^+_{\max}) \subset ( D \otimes_{alg} F)^+_{\max},$$
 where
 \begin{equation}\label{al}
 ( D \otimes_{alg} E)^+_{\max}= ( D \otimes_{alg} E) \cap  (D \otimes_{\max} A)_+\end{equation}
 and similarly for $( D \otimes_{alg} F)^+_{\max}$. 

 Then (see Remark \ref{nor})
$Id_D \otimes
 u$ is bounded with respect to the max-norms   and hence we have $u\in MB(E,F)$. 
 Whence $Id_D \otimes
 u$  extends (by density)
 to a (bounded) positive map 
 $$ u_D: D \otimes_{\max} E \to D \otimes_{\max} F.$$
 We denote by $MP(E,F)$ the set of such maps.\\
 Since we may replace   $D$ by $M_n(D)$ 
 the map $ u_D$ is actually c.p. and even m.p. whence (see Remark \ref{nor}) $\|u\|_{mb}=\|u_D(1\otimes 1)\|=\|u(1)\|$.
 
Recapitulating:
\begin{pro}\label{mpp} Let  $u:E \to F$ be a linear map between   operator systems $E\subset A$ and $ F\subset   B$ 
with $A,B$ denoting  unital $C^*$-algebras.
The following are equivalent:
 \begin{itemize}
                  \item[{\rm (i)} ] 
 $u\in MP(E,F)$. 
 \item[{\rm (ii)} ] The maps
$u_D$ are positive for any $C^*$-algebra $D$. 
 \item[{\rm (iii)} ] The maps
$u_D$ are c.p. for any $C^*$-algebra $D$. 
 \item[{\rm (iv)} ]  The maps
$u_D$ are m.p. for any $C^*$-algebra $D$.
 \end{itemize}
In that case $\|u\|_{mb}=\|u(1)\|$.

\end{pro}

 Obviously m.p. maps are c.p.
  The converse (c.p. implies m.p.) holds when $E=A$  but
   is {\it no longer true} in general when $E$ is merely an operator system in $A$ (see Remark \ref{R5}).
   This is an important point.

   In analogy with the m.b. case (see \eqref{e10''}), we may state:
   \begin{pro}
   Consider operator systems $E_j\subset A_j$, $F_j\subset B_j$ ($j=1,2$).
    For any pair of  m.p. maps $u_j: E_j \to F_j$, the map $u_1\otimes u_2$
    extends to a map in $MP(E_1\otimes_{\max} E_2, F_1\otimes_{\max} F_2)$.      

   \end{pro}
 
 It is known that for c.p. maps on a $C^*$-algebra $A$ the mb-norm
 is equal to the ordinary norm (and a fortiori to the cb-norm).
 More generally, we have:
 
\begin{pro}\label{L1} Let  $A,B$ be $C^*$-algebras.
Any c.p. map $u: A\to B$ is m.b. with $\|  u \|_{mb}=\|u\|$, and $\|  u \|_{mb}=\|u(1)\|$  if    $A$ is unital. In particular $\|  u \|_{mb}=1$  for any u.c.p. map $u$.
More generally we have
$D(A,B)\subset MB(A,B)$
and  $$\forall u\in D(A,B) \quad \|  u \|_{mb}\le  \|  u \|_{dec}.$$  
  \end{pro}
  \begin{proof} 
  This follows   from Haagerup's definition of the dec-norm
    once one recalls Proposition \ref{imp}.
   \end{proof}
\begin{cor}\label{c1} In the preceding situation,  
$CP(A,B)= MP(A,B).$
\end{cor}
It is important to emphasize for the reader that
Proposition  \ref{L1}  (or its corollary) is {\it not valid} for a map
$u$ only defined on an operator subsystem of $A$.
See Remark \ref{R5} below for clarification.

The following statements  are variants of Kirchberg's characterization
of m.b. maps as  decomposable ones. 
They first appeared in \cite[p. 266]{P4} 
(and also in \cite[\S 6.3]{154})  with his permission.  

Together with Proposition  \ref{L1}, they  allow us to view
m.b. maps  as some kind of generalization of decomposable maps.

\begin{thm}\label{d1}     For any map $u: E \to M$ from an operator space
 $E\subset A$ to a von Neumann algebra  $M$ we have
$$\|u\|_{mb} =\inf \{ \|\tilde u: A \to M\|_{dec} \mid \ \tilde u \in D(A,M),\ \ \tilde u_{|E}= u\}, $$
and the infimum is attained.
In particular, $D(A,M)=MB(A,M)$ and $\|  u \|_{mb}=  \|  u \|_{dec}$ for all $u:A \to M$.\\
If $E$ is an operator system and $u\in MP(E,M)$ 
then we can find an extension $\tilde u\in MP(A,M)$.

\end{thm}

\begin{proof} For the first part see \cite[p. 266]{P4} 
(and also in \cite[\S 6.3]{154}). 
Since $M$ is a dual we can use the weak* topology to show that the infimum is attained. Let $u\in MP(E,M)$. We will use Remark \ref{nor}. Let $E$ be an operator system. 
Assume  $u(1)=1$ and hence $\tilde u(1)=1$ and   $\|u\|_{mb}=1$. If we choose
$\tilde u$ such that $\|  \tilde u \|_{mb}= 1$ and a fortiori $\|  \tilde u \|_{cb}= 1$ then 
$\tilde u : A \to M$ is c.p. and hence m.p. by the preceding corollary. 
This settles the unital case. For the non-unital case on $E$ unital, we will use a well known perturbation argument.
We fix $\vp>0$ and we choose any state $f$ restricted to $E$. Let $u_\vp(.)=u(.)+\vp f(.) 1$. We define $v_\vp\in MP( E , M)$ by $v_\vp(.)=   u_\vp(1)^{-1/2} u_\vp(.)u_\vp(1)^{-1/2}$. Since $v_\vp$ is unital, by the unital case,  
it admits  an extension $\tilde v_\vp\in MP(A , M)$. Then, let $\tilde {u_\vp}: MP(A,M)$ be defined by
$\tilde {u_\vp}(.)=   u_\vp(1)^{1/2} \tilde v_\vp(.)u_\vp(1)^{1/2}$. We have
$\| { {\tilde {u_\vp}  }  } _{|E} -u\| \le \vp$. Let $\tilde u: A \to M$ be a pointwise
weak* limit of a suitable ultrafilter refining the net $({\tilde {u_\vp}  })$. The map
$\tilde u\in MP( A , M)$ is the desired extension of $u$.
\end{proof}

\begin{thm}\label{d2}  Assuming $E\subset A$ and $ F\subset   B$ as before, a map
$u: E \to F$ is m.b. if and only if
there is a decomposable map $\tilde u : A \to B^{**}$ with  $\|\tilde u \|_{dec}
=\| u \|_{mb}$
that extends the map $i_F u: E \to F^{**}$ 
where $i_F :F \to F^{**}$ is the canonical inclusion map, and
$F^{**}$ is viewed as a subspace of $B^{**}$.\\
The map $u: E \to F$ is m.p. if and only if there is a c.p.   map $\tilde u : A \to B^{**}$
extending $i_F u$.
\end{thm}
\begin{proof} For the ``only if" part,  we simply apply the preceding statement to  $i_F u: E \to B^{**}$. For the ``if part" recall that $i_B: B \to B^{**}$ is a max-isometric $*$-homomorphism, and hence
$\|i_F u\|_{mb} =\|  u\|_{mb} $. Similarly, since 
$D \otimes_{max} B \subset D \otimes_{max} B^{**} $
 for any $C^*$-algebra $D$,
we have
obviously 
$(D \otimes_{max} B^{**})_+ \cap (D \otimes_{max} B ) =(D \otimes_{max} B )_+$
and hence 
using the notation \eqref{al}
$ (D \otimes_{alg} B^{**})^+_{\max} \cap (D \otimes_{alg} B)  =(D \otimes_{alg} B)^+_{\max}$ and therefore
$ (D \otimes_{alg} B^{**})^+_{\max} \cap (D \otimes_{alg} F)  =(D \otimes_{alg} F)^+_{\max}$.
\\
Thus  if $\tilde u : A \to B^{**}$ is c.p.
it is m.p. by  Corollary \ref{c1} and its restriction to $E$ is m.p.
\end{proof} 

\begin{rem}\label{R5-} Let $u\in MB(E, F)$ for
operator systems $E\subset A$ and $ F\subset   B$. 
For any $C^*$-algebra $D$, 
let $u_D: D\otimes_{\max} E \to  D\otimes_{\max}  F$ be 
the map extending $Id_D \otimes u$.
By \eqref{e10'} we know that    $\|u_\C\|  \le 1$ (or equivalently $\|u\|_{mb} \le 1$)   if and only if  
$\|u_D \|\le 1$ for all $D$.\\
In analogy with the latter fact
we claim that  if $u_\C$ is positive then $u_D$ is positive  for all $D$, or equivalently $u$ is m.p.
by Proposition \ref{mpp}.\\
To check this claim it is easy to see that we may assume $D$ unital.
 Assume first that
  $u_\C$ is (well defined and)  c.p.
  Then the case when $u$ is also unital is easy using Remark \ref{nor} :   we have $\|u_\C\|_{cb}=1$
  and a fortiori $\|u_\C\| =1$. Now since $u_D$ is unital
     $\|u_D\|=1$  by   \eqref{e100}    and hence (again by Remark \ref{nor}) $u_D$ is    positive.
     This proves our claim when $u$ is unital.
       In the case when $u$ is  non-unital  
 we may argue as in the proof of Theorem \ref{d1}: without loss of generality we may  assume $F=B $, we introduce the unital map $v_\vp: E \to B$. Our assumption that $u_\C$ is c.p. implies that  $[v_\vp]_\C$ is c.p. The unital case tells
 us that  $[v_\vp]_D$ is c.p. for any $D$.  Letting $\vp\to 0$ we conclude
 that $ u_D$ is c.p. and hence  that $u$ is m.p. This completes the proof of the claim when we assume
 that $u_\C$ is   c.p.  Now assume $u_\C$  positive.
 Since $M_n(\C)$ has the lifting property (see \S \ref{lift}) for any $n\ge 1$, we have a c.p. factorization of the identity of $M_n(\C)$ through $\C$. It is easy to deduce from this  that $u_\C$ is   c.p.
  \end{rem}

\begin{rem}\label{R5}
Let $E\subset A$ be an operator system, $C$ a $C^*$-algebra. 
While c.p. maps are m.p. when $E=A$, this is not valid
for maps defined on a general operator subsystem.  
Indeed, if $E\subset A $ is a $C^*$-subalgebra,
we may view it as an operator subsystem and then
the identity $u: E \to E$ is trivially c.p. but it is m.p.
only  if $u$ is m.b. and this means that the
max-norm induced on $\C \otimes  E$ by $\C \otimes_{\max} A$
dominates the intrinsic max-norm of the $C^*$-tensor product
of $\C$ and $E$, which is not true in general, when the inclusion is not
max-isometric.
\end{rem}

\begin{rem}\label{mazur}
In several instances  when we need to  pass from the bidual $C^{**}$
to the space $C$  itself we invoke ``Mazur's theorem". By this we mean 
the famous theorem (a variant of Hahn-Banach)  by which for convex subsets 
of a Banach space $C$ the weak closure (meaning $\sigma(C,C^*)$)
and the norm closure coincide. We usually exploit the observation that
the topology $\sigma(C^{**},C^*)$ {\it coincides on $C$} with the weak topology.
\end{rem}

\begin{rem}\label{R15} Let 
$u: E \to B$ be a linear map from  an o.s.  $E\subset A$  to   a $C^*$-algebra $B$.
Since the inclusion map $i_B: B \to B^{**}$ is max-isometric (see Remark \ref{R1}),
$u\in MB(E,B) \Leftrightarrow i_Bu \in MB(E,B^{**})$ and
we have
 \begin{equation}\label{e40}
\|i_Bu\|_{mb} =\|u\|_{mb}.
\end{equation}
Note that this is not valid in general for dec in place of mb.\end{rem}

\section{Maximal local reflexivity}\label{LR}

In this preliminary section, we analyze the consequences of the validity
of a form of ``local reflexivity" for the maximal tensor product.
Given a f.d. subspace $E\subset A$
of a $C^*$-algebra $A$ and  another $C^*$-algebra $C$,
we have a natural identification 
between $MB(E,C)^{**}$ and $MB(E,C^{**})$
for which it is easy to see that 
$\|MB(E,C)^{**} \to MB(E,C^{**}) \|=1$.
We need to single out the case when
this map is actually isometric, i.e. when we have 
\begin{equation}\label{e12}
MB(E,C^{**})=MB(E,C)^{**} \ \ \text{isometrically}.
\end{equation}
One could also measure the deviation from the isometric case
by a constant, but we will not need this here.

We will  state first a basic  extension property,
then a basic lifting property from \cite{157}. Both are consequences of
\eqref{e12}.
More on these themes
in \S \ref{ext} and \S \ref{lift}.

\begin{thm}\label{d3}
Let $X\subset A$ be a separable o.s. and $C$ a $C^*$-algebra.  Assume that \eqref{e12} holds for any f.d. $E\subset X$.
Then for any $\vp>0$, 
any $u: E \to C$ admits an extension $\tilde u : X \to C$
with
$\|\tilde u\|_{mb} \le (1+\vp) \|u\|_{mb}$.
\end{thm}

\begin{proof} See \cite[Th.3.5]{157}.
\end{proof}

\begin{thm}\label{d4} Let $E\subset A$ be a f.d. o.s. satisfying \eqref{e12}.
Then for any quotient $C/\I$ of $C$,
any $u\in MB(E,C/\I)$ admits a lifting $\hat u\in MB(E,C)$ 
with $\|  \hat u\|_{mb}=\|  u\|_{mb}$. Here lifting means
that if $q: C \to C/\I$ denotes the quotient map we have $q\hat u=u$.
\end{thm}

\begin{proof} See \cite[Remark 4.2]{157}. The proof is the same as that of (iv) $\Rightarrow$ (i)
in \cite[Th. 2.6]{157}  applied to the class
 ${\cl F}$ equal to the class of maps with mb-norm $\le 1$.
\end{proof}

\section{The mb-norm and the lifting property (LP)}\label{lift}

\begin{dfn} A unital $C^*$-algebra $A$ has the lifting property (LP)
if any u.c.p. map from $A$ to a quotient $u: A \to C/\I$ admits
a u.c.p. lifting 
$\hat u: A \to C$. If $A$ is not unital we say that it has the LP if its unitization does.  This is equivalent to
requiring that any completely contractive (in short c.c.) c.p. map
$u: A \to C/\I$
admits a c.c.c.p. lifting $\hat u: A \to C$.
\end{dfn}

\begin{dfn}\label{D1} An o.s.   $X\subset A$ has the mb-lifting property (MLP)
if any mb map from $X$ to a quotient $u: X \to C/\I$ admits
a mb lifting $\hat u: X \to C$ with $\|\hat u\|_{mb}= \|u\|_{mb} $.
\end{dfn}
\begin{rem}\label{arv} By Arveson's principle from \cite{Arv}  (pointwise limits of liftables are liftables, see   \cite[\S 2]{157} for a more general formulation),
for the MLP with $X$ separable it suffices to check
 that for any $\vp>0$ any $u: X \to C/\I$ admits
a  lifting $\hat u: X \to C$ with $\|\hat u\|_{mb}\le (1+\vp) \|u\|_{mb} $.
\end{rem}

The next result was essentially proved in \cite{157} for $X=A$. The
proof for a subspace $X\subset A$ is identical. However,
the equivalence with (ii) , (v), or (vi)    is a new observation
which somewhat  surprisingly
 shows that {\it in the mb-context} local lifting 
 (and even mere local embeddability in $\C$) implies global lifting.
 
\begin{thm}\label{d6} Let $X\subset A$ be a separable o.s.
The following are equivalent:
 \\
 {\rm (i)} The space $X$ has the MLP.\\
 {\rm(ii)} For any f.d. subspace $E\subset X$ and any $\vp>0$  there is a subspace $\hat E \subset \C$
such that $d_{mb}(E,\hat E)\le 1+\vp$.\\
{\rm(ii)'} Same as (ii) with $\vp=0$.\\
{\rm (iii)}  For any family $(D_i)$ of $C^*$-algebras, the canonical
(norm one) map $$\ell_\infty(\{D_i\}) \otimes_{\max} X\to \ell_\infty(\{D_i \otimes_{\max} X\}) $$
is isometric.
\\ 
 { \rm (iv)} Same as (iii) with $I=\NN$ and $D_i=\C$ for all $i$. 
\\
 {\rm(v)} For any f.d. subspace $E\subset X$ and any $C^*$-algebra
$D$ the canonical isomorphism
$$ MB(E, D)^{**} \to MB(E, D^{**}) $$
is  isometric.
\\
{\rm(vi)} Any f.d. subspace $E\subset X$ has the MLP.
\\
Moreover, when $X=A$ (separable) these are equivalent to \\
 {\rm (i)'} The $C^*$-algebra $A$ has the LP.\\

\end{thm}
\begin{proof} Since the inclusion of $A$ in its unitization is max-isometric we may assume $A$ unital.
Since the proof is similar to that in \cite{157}
we only give a sketch.\\
(i) $\Rightarrow$ (ii) : we use $A=C^*(G)/ \I$ for some large enough free group $G$. Then we lift the restriction to $E$ of the map $u: X \to C^*(G)/ \I$.
This gives us $\hat u : E \to \hat E\subset C^*(G) $ with $\|\hat u\|_{mb} =1$
lifting $u_E$. Reducing the range of $\hat u$ we may assume that $\hat u$ 
is bijective, and then passing to a suitable countable free subgroup of $G$
we obtain  $\hat E\subset \C$ and the restriction of the quotient map
$C^*(G) \to C^*(G)/ \I$ gives us $(\hat u)^{-1}$ with $\|(\hat u)^{-1} \|_{mb} \le 1$, whence (ii)' and a fortiori (ii).\\
(ii) $\Rightarrow$ (iii) : By \cite{157} the algebra $\C$ satisfies (iii). Using (ii)
it is easy to check that any f.d. $E\subset X$ also satisfies  (iii). Equivalently
$X$   satisfies  (iii). \\
(iii) $\Rightarrow$ (iv) is obvious and (iv) $\Rightarrow$ (iii) follows from an elementary reduction to the separable case, together with the fact that any separable $D$ is a quotient of $\C$.\\
(iii) $\Rightarrow$ (v) is proved 
 using the description    of the unit ball of the dual of $MB(E,D)$ from \cite{157}.\\
 (v) $\Rightarrow$ (i) :
     Let $u: X \to C/\I$ as in Definition \ref{D1}. Theorem \ref{d4}  tells us that $u_{|E}: E \to  C/\I$ admits
     a lifting $v: E \to C$ with $\|v\|_{mb}=\|u_{|E}\|_{mb} \le \|u\|_{mb}$, then
     Theorem \ref{d3} tells us that $v$ admits an extension $\tilde v: X \to C$ with
     $\|\tilde v\|_{mb}\le (1+\vp)\|v\|_{mb}$. Recall $q: C \to C/\I$ denotes the quotient map.
      Let $w^{E,\vp}= (1+\vp)^{-1}  q \tilde v : X \to   C/\I$ (which is obviously liftable by $(1+\vp)^{-1}    \tilde v$). Note that $w^{E,\vp}$ is equal to $(1+\vp)^{-1} u$ on the subspace $E$.
      From this follows that $u$ is the pointwise limit
     of a net of maps indexed by $(E,\vp)$ (namely the maps $w^{E,\vp}$) that are all liftable by maps
     with mb-norms $\le \|u\|_{mb}$. By Arveson's principle (see Remark \ref{arv}), it follows that
     $u$ itself is liftable by a map
     with mb-norm $\le \|u\|_{mb}$, and we obtain (i).\\
          This gives us the equivalence of the five first assertions. From the equivalence
     of (i) and (ii) (or (v)) it is clear that (i) and (vi) are equivalent.
   \\  
     Lastly, if $X=A$ (separable) it is proved in \cite[Th.0.13]{157} that (iii)  is   equivalent to the LP of $A$.
 \end{proof}
 \def\B{B(\ell_2)}
  \begin{rem} In analogy with (ii) a $C^*$-algebra $A$ has the LLP
  (meaning that $A\otimes_{\min} \B= A\otimes_{\max} \B$) if  and only if 
  for any f.d. $E\subset A$ and any $\vp>0$ there is a subspace $\hat E \subset \C$ and a linear isomorphism $v: E \to \hat E$ such that  
$ \|v\|_{cb} \|v^{-1}\|_{mb}\le 1+\vp  $ (note that here the two norms differ). Indeed, one way to check this
is to observe that $\C$ has the LLP and that using the above $v$'s one shows easily that
the LLP of $\C$ passes on to $A$.

  \end{rem}

 \begin{cor} Let $  A$ be a separable $C^*$-algebra with LP. Then any subspace of $A$
 has the MLP. In particular, any subspace of $\C$ has the MLP.
 \end{cor}
   
   \begin{rem}  For certain properties of operator spaces it is often useful to consider variants of
   the property associated to a constant $c\ge 1$. This has been fruitful
   for instance  for $c$-exactness (see \cite[p. 286]{P4}) of for $c$-local reflexivity (see \cite[p. 303]{P4}). 
   Analogously, in the situation of Definition \ref{D1} let us say that $X$ has the $c$-MLP
   if   any $u$ admits a lifting $\hat u$ such that $\|\hat u\|_{mb} \le c \|u\|_{mb}$.
   It is easy to check that Theorem \ref{d6} remains valid for the $c$-MLP, 
   once each of  the  equivalent conditions is suitably adapted. In particular
   $X$ has the $c$-MLP if and only if for any f.d. subspace $E\subset X$ there is
   $\hat E\subset \C$ such that $d_{mb}(E,\hat E)\le c$, or if and only if the map
   $\Phi$ appearing in (iii) is $c$-isometric,  by which we mean that $\|\Phi(.)\|\ge c^{-1} \|.\|$.
   In the special case when $X=A$, since the latter map is a $*$-homomorphism 
   the $c$-MLP is equivalent to  the $1$-MLP (i.e. the MLP) and hence to the LP.
 
 \end{rem}

   We now briefly consider the analogue of the MLP
when c.b. replaces m.b. This is  Ozawa's  OLP
which is much stronger than the LP. For instance
$M_n$ or $\ell_\infty^n$ for $n\ge 3$ fail the OLP.
 We refer to Ozawa's \cite{Oz} or \cite[p. 278]{P4} for more information.

\begin{dfn} An operator space  $X\subset A$ has the OLP
if any c.b. map from $X$ to a quotient $u: X \to C/\I$ admits
a c.b. lifting 
$\hat u: X \to C$ with the same c.b. norm.     \end{dfn}

    The point of the next result is to show that $CB(X, .) =MB(X,.)$  is rather restrictive.
 
\begin{pro} Let $X\subset \C$. 
                      Assume $CB(X, C) =MB(X,C)$ isometrically for any $C$.
                      Then $X$ has the OLP.
    \end{pro}

                      \begin{proof}  
                      Since any  $X\subset \C$ has the MLP,
                      cb-maps from $X$ to quotients have cb-liftings
                      as required by the OLP.
                      \end{proof}
 \begin{rem} It seems open whether a non-separable $C^*$-algebra 
 with MLP has the LP.  Actually, it seems to be open whether
 $C^*(G)$
 has the LP when $G$ is an uncountable free (or even free abelian !) group.
 Nevertheless it is easy to check that the equivalence of (i)-(vi) in
Theorem \ref{d6} remains valid when $X$ is non-separable provided we replace (i) by 
the property that  
  every separable subspace of $X$ has the MLP.
  \end{rem}

\section{Extension properties and the WEP}\label{ext}

A $C^*$-algebra $A\subset B(H)$ has the weak expectation property
(WEP in short) in Lance's sense if any linear map $u: A \to M$ 
with values in a von Neumann algebra $M$ admits an extension
$\tilde u: B(H) \to M$ with $\|\tilde u\|_{cb}=\|u\|_{cb}$.

\begin{rem}\label{kir}
Kirchberg   (see \cite[p. 189]{154}) proved that $ A$ has the WEP if and only if
the min and max norms coincide on $A\otimes \C$. In particular this happens whenever $A$ is injective.
\end{rem}

We start by
a reformulation  of Theorem \ref{d1} for the case of m.b. maps with values in a von Neumann algebra.
  
\begin{thm}\label{ext1} Let $X\subset A$ be an o.s. and $M$ a von Neumann algebra .
Then any m.b. map $u: X \to M$ admits
an m.b. extension  $\tilde u: A \to M$
extending $u$ with $\|\tilde u\|_{mb} =\|  u\|_{mb}  $.
\end{thm}

In parallel the case of maps with values in a $C^*$-algebra requires an assumption on $X$.

\begin{thm}\label{ext2} Let $E,X$ be   o.s. with $E\subset X\subset  A$  and let $C$ be a $C^*$-algebra.
Assume $\dim(E)<\infty$ and $X$ separable with the MLP.
Then, for any $\vp>0$,   any m.b. map $u: E \to C$ extends to an  
 m.b. map $\tilde u: X \to C$ with
\begin{equation}  \label{e11} \|\tilde u\|_{mb} \le (1+\vp) \|  u\|_{mb}  .\end{equation} 
\end{thm}

\begin{proof}  
 By  (i) $\Rightarrow$ (v)  in Theorem \ref{d6}, Theorem \ref{ext2} is a reformulation of      Theorem \ref{d3}.
\end{proof}

\begin{rem} Even if we assume $C$ with WEP, Theorem \ref{ext2} is  probably not valid  
if we drop the MLP assumption on $X$, but this is a somewhat {\it sticky point}.
Indeed,   
when $C$ has the WEP, the mb-norms of $C$-valued maps are the same as their cb-norms. So if  Theorem \ref{ext2} holds without the MLP assumption on $X$ for a $C^*$-algebra $C$
then the extension property expressed by \eqref{e11}
is valid with cb-norms instead of mb-norms and for $A=B(H)$.
In that case one says (following Effros-Haagerup)  that $C$ is ``approximately injective".
Kirchberg in \cite{Ki} claims that if  the WEP implies  approximately injective  then
  the WEP implies the LLP,   and the latter implication was disproved in \cite{JuP}.
However, as pointed out in the footnote of \cite[p. 362]{154} the argument for this claim is
unclear, and a convincing  proof that
WEP $\not\Rightarrow$ approximately injective  is still missing.
\end{rem}

\begin{lem}\label{L6} Let $E\subset A$ be a f.d. operator system, $C$ a unital $C^*$-algebra.
Assume that the $C^*$-algebra $A$ has the MLP.
For any m.p. map $u: E \to C^{**}$ with $\|u\|\le 1$ such that $u(1)\in C$,
there is a net of m.p. maps
$u_i: E \to C $ with $\|u_i\|\le 1$ tending pointwise
weak* to $u$.
\end{lem}

\begin{proof} For unital maps this is Corollary 3.9. in \cite{157}.
The case when $u(1)\in C$ follows by a well known trick.
Let $\vp>0$ and let $f$ be a state on $E$. Let
$u_\vp(.)= u(.)+ \vp f(.) 1$. Note
$u_\vp(1)\in C$ and is invertible. Let
$v_\vp(.)= u_\vp(1)^{-1/2} u_\vp(.) u_\vp(1)^{-1/2}$ so that $v_\vp(1)=1$.
Applying the unital case,
we have a net $v_i$ of $C$-valued unital m.p. maps tending pointwise
weak* to $v_\vp$. Then
$u_i(.)=u_\vp(1)^{1/2} v_i(.) u_\vp(1)^{1/2}$ tends weak* to $u_\vp$.
Since $\lim_{\vp\to 0}\|u-u_\vp\|=0$ the announced result follows.
\end{proof}

The next statement makes more precise the relationship between
the mb maps and the decomposable ones in the MLP context.

\begin{thm}\label{T1} Let  $  A,C $ be  $C^*$-algebras. Assume $A$ has the MLP and let $E\subset A$ be a f.d. 
operator system.
Let $u: E \to C$. The following are equivalent.
  \item{(i)} We have $\|u\|_{mb} \le 1$.
    \item{(ii)} For any f.d. operator system $E_1$ with $E\subset E_1 \subset A$ and any $\vp>0$
    there is a map $  u_1: E_1 \to C$ extending $u$ and such that $\|u_1\|_{dec} <1+\vp$.
\end{thm}
\begin{proof} Assume (i). By Theorem \ref{d2}  (note $E_1$ has the MLP) there is 
$v: E_1 \to C^{**}$ extending $i_C u$ such that $\|v\|_{dec}\le \|v\|_{mb}\le 1$.
Whence  $C^{**}$-valued {\it unital} c.p. maps 
$S_1,S_2$
such that
$V: x\mapsto \left(\begin{matrix}
 S_1 (x) &  v(x)\\
 v(x^* )^* & S_2 (x)
\end{matrix}\right)$
is in $CP(E_1, C^{**})$ (see Remark \ref{R6}). Note 
$v(1)\in C$   since $v_{|E}=u$ and hence $V(1)\in M_2(C)$.
By Lemma \ref{L6} there is a net of m.p. maps $V_i: E_1\to M_2(C)$
with $\|V_i\|\le 1$ tending pointwise weak* to $V$.
In particular $(V_i)_{12} $ restricted to $E$ tends pointwise weak* to $ u$.
Passing to suitable convex hulls and using Mazur's theorem
we can find (possibly different) c.p. maps $V_i$ with $\|V_i\|\le 1$  for which
$(V_i)_{12} $ restricted to $E$ tends pointwise in norm to $ u$.
Choose then $i$ so that $\|  u - {(V_i)_{12}}_{|E} \|  <\vp$.
Let then $v'=  (V_i)_{12} :E_1 \to C$. We have $\|v'_{|E}-u\|< \vp$
  and 
$\|v'\|_{dec} \le 1 $ (since $\|V_i\|\le 1$). By the elementary perturbation argument
described in \cite[Rem. 3.4]{157} there is a map $v'': E_1\to C$ such that
$v''_{|E}=u$ with $\|v''\|_{dec} \le 1+f(\vp)$ where $f(\vp)\to 0$ when $\vp\to 0$. 
Thus we obtain (ii).\\
 Assume (ii). By a simple ultraproduct argument involving the net of maps
 $  u_1: E_1 \to C$ (indexed by f.d. subspaces $E_1$ with $E\subset E_1 \subset X$ and $\vp>0$)
 we obtain a map $\tilde u: A \to C^{**} $ with $\|\tilde u\|_{dec} \le 1$
 extending $i_C u: E \to C^{**} $. Equivalently $\|\tilde u\|_{mb} \le 1$ (see Theorem \ref{d1})
 and a fortiori $\|  i_C u \|_{mb}= \|\tilde u_{|E}\|_{mb} \le 1$. By \eqref{e40} $\|    u \|_{mb}\le 1$.
\end{proof}
We recall that a $C^*$-algebra $C$ has the WEP (in Lance's sense) if  and only if
 the min and max norms coincide on $\C \otimes C$  (see \cite[\S 9.3]{157}). By Kirchberg's fundamental theorem
 (see \cite[Th. 9.6]{157}) it is so when $C=B(H)$ or when $C$ is injective.\\
It is natural to wonder when the equality
$CB(E,C)=MB(E,C)$ holds isometrically for any $E$.
It is immediate that it does so if $\C\otimes_{\min} C= \C\otimes_{\max} C$,
or equivalently if $C$ has the WEP.
The converse is also true as the next statement shows. Note that in sharp contrast with
Theorems \ref{T1} and \ref{ext2}
 the assumption is on the range and not on the domain space.

\begin{thm}\label{wep} The following properties of a $C^*$-algebra  $C$ are equivalent.\\
{\rm (i)}  $C$ has the WEP.\\
{\rm (ii)} For any $C^*$-algebra $A$ and any o.s. $E\subset A$  we have
$CB(E,C)=MB(E,C)$   isometrically.\\
{\rm (iii)} For any pair $E ,F$ of f.d. o.s. with $E \subset F\subset \C$, any $\vp>0$ and
 any $u\in CB(E,C)$    there is   $\tilde u: F \to C^{**}$ with 
$\|\tilde u\|_{mb} \le (1+\vp)  \|  u\|_{cb}$ such that ${\tilde u}_{|E}=i_C u$.\\
{\rm (iv)} For any pair $E ,F$ of f.d. o.s. with $E \subset F\subset \C$, any $\vp>0$ and
 any $u\in CB(E,C)$    there is   $\tilde u: F \to C $ with 
$\|\tilde u\|_{cb} \le (1+\vp)  \|  u\|_{cb}$ such that ${\tilde u}_{|E}=  u$.\\
{\rm (v)} For any f.d. $E\subset \C$ and any $\vp>0$, any c.b. $u: E \to C$ admits an extension
$\tilde u: \C \to C$ with $\|\tilde u\|_{cb} \le  (1+\vp)  \|  u\|_{cb}$.\\
{\rm (vi)} Same as (iv) with $F$ restricted to be a subspace spanned by  a finite number of the free unitary generators of $\C$.
\end{thm}
\begin{proof} 
We first give the simple  proof that (i) $\Leftrightarrow$ (ii), then   outline the proof of 
 the other equivalences.\\
(i) $\Rightarrow$ (ii) is obvious as we just observed. 
Conversely, if (ii) holds, we have $\|u\|_{mb}=\|u\|_{cb}$ for any $u: E \to C$. Assuming $C\subset B(H)$ we may apply this with
$E=C$,  $A=B(H)$, and $u(x)=x$ for $x\in E$. Then  $\|u\|_{mb}=\|u\|_{cb}=1$ tells us that 
the max-norm on $\C \otimes C$ is $\le $ the norm induced on it by the max-norm
on $\C \otimes B(H)$, which is equal to the min-norm by Kirchberg's theorem.
So the latter induced norm on $\C \otimes C$ is also the min-norm.  It follows that
the max-norm and the min-norm are identical on $\C\otimes C$.
Equivalently $C$ has the WEP, proving  (ii) $\Rightarrow$ (i). \\
Assume (ii). By Theorem \ref{ext1} any $u\in CB(E,C)=MB(E,C)$ admits an ``extension"
$\tilde u\in MB(F,C^{**})$ with $  \|\tilde u\|_{mb} \le (1+\vp)  \|  u\|_{cb}$,
whence (iii). Assume (iii). 
By   Theorem \ref{d6}, there is a net $(\tilde u_i)$
in the ball of radius $1+\vp$ in  $MB(F,C )$ tending
pointwise weak* to $\tilde u$. A fortiori $\|\tilde u_i\|_{cb} \le 1+\vp$.
Passing to convex combinations using Mazur's theorem
one gets a net $(u_i)$ in the same cb-ball such that the restrictions $({u_i}_{|E})$
converge pointwise in norm to $u$. Then a perturbation argument (as in the proof of Theorem \ref{T1})  leads to (iv). Assume (iv). Fix a f.d. $E\subset \C$ and  $u\in CB(E,C)$. 
Let $E=E_0 \subset \cdots E_n\subset E_{n+1}\subset \cdots  \subset \C$ be a sequence of f.d. subspaces with dense union $V=\cup E_n$. Let $\vp_n>0$
be such that $\prod (1+\vp_n) <1+\vp$. By applying (iv)  to each $E_n$ we find
a map $\tilde u: V \to C$ with $\|\tilde u\|_{cb} \le 1+\vp$ extending $u$.
We may then extend $\tilde u$ to the whole of $\C$ by density of $V$,
whence (v). Clearly (v) implies (iv) which implies (vi).
By \cite[Th. 21.3]{157} 
 (vi) and (i) are equivalent.
\end{proof}

                          \section{Duality}\label{dual}

                     The o.s. dual of an o.s. $X\subset A$ can be defined by an isometric embedding $f: X^* \to B(H)$ such that for any o.s. $Y$ and 
                     $s\in  Y \otimes_{alg} X^*  $ we have
                     $$\|s\|_{\min}=\|u_s\|_{cb}$$
                     where $u_s : X \to Y$ is the  linear map 
                     canonically associated to $s$, and the embedding $f$ is used
                     to define the min-norm of $s$. Thus, denoting by 
                     ${F(X,Y)}$ the subspace of finite rank maps in $CB(X,Y)$,
                     we have an isometric identity
                     $$Y \otimes_{\min} X^* =\ovl{F(X,Y)}^{cb} \subset CB(X,Y).$$

                      Since   duality  plays an important role
                       in o.s. theory it is natural to wonder what remains of it
                       in the ``maximal" context. While
                       the following naive statement seems less efficient (because of the restriction $Y\subset \C$), it does give
                       a partial duality in our new context.
                       
                     \begin{thm}
                     Let $X\subset A$ be a subspace of a $C^*$-algebra $A$, with
                     dual $X^*$ as above.
                     There is a (completely) isometric embeddding $X^*\subset L$ in a $C^*$-algebra $L=L_{X\subset A}$ (of the form $\ell_\infty(I,\C)$ for some suitably large index set $I$) such that for any $s\in  \C \otimes_{alg} X^*  $ with associated linear map 
                     $u_s : X \to \C$ we have
                     $$ \|s\|_{\max}=\|u_s\|_{mb} ,$$
                     and hence for any $Y\subset \C$  we have an isometric identity
                      $$Y \otimes_{\max} X^* =\ovl{F(X,Y)}^{mb} \subset MB(X,Y).$$
                     
                       \end{thm} 
                       \begin{proof}
                       Let $I$ denote the unit ball of $\C \otimes_{\max} X$.
                       Any $t\in I$ defines a linear map $v_t: X^* \to \C$.
                       Let $j: X^* \to \ell_\infty(I,\C)$ be the map taking
                       $\xi \in X^*$ to $(v_t(\xi))_{t\in I} $. Since $\|v_t\|_{cb} =\|t\|_{\min} \le \|t\|_{\max} \le 1$ we have $\|j\|_{cb}\le 1$. 
                       Moreover since
                       $M_n$ has the MLP (and is a quotient of $\C$), there is a max-isometric imbedding $M_n\subset \C$   for any $n$, so we may use the $v_t$'s with range in such copies of $M_n$ to show that
                         the map $j$
                       is completely isometric on the dual $X^*$
                       (in the o.s. sense) of $X$.\\
                       Let $s\in  \C \otimes_{alg} X^*$.
                       Then since $\C$ has the LP we know
                       that $\C \otimes_{\max} \ell_\infty(I,\C) \subset 
                       \ell_\infty(I, \C \otimes_{\max} \C)$ isometrically
                       (see (iii) in Th. \ref{d3}) and hence
                       $$\|s\|_{\max} = \sup_{t\in I} \|[Id \otimes v_t](s) \|_{\max}= \sup_{t\in I}\|[Id \otimes u_s](t) \|_{\max}
                        $$
                       where for the last equality we used the
                      fact that  $[Id \otimes v_t](s)$ and $[Id \otimes u_s](t)$ coincide
                      after flipping the factors.
                       Whence  $ \|s\|_{\max}=\|u_s\|_{mb} $. Since (by definition)
                       $Y \otimes_{\max} X^* \subset \C \otimes_{\max} X^*$
                       the rest is clear. \end{proof}
                       
\begin{rem} Let $X\subset A $ and $Y\subset B$ be o.s.
with duals $X^*\subset L_{X\subset A}$  and $Y^*\subset L_{Y\subset B}$
as above. 
For any $u\in MB(X,Y)$, we have $\|u^*: Y^* \to X^*\|_{mb} \le \|u\|_{mb}$.
Indeed, this boils down to the fact that
for any $w: Y \to \C$ of finite rank we have $\|wu\|_{mb} \le \|w\|_{mb} \|u\|_{mb}$.\\
More precisely, 
let  $R(u): MB(Y,\C) \to MB(X,\C)$ be defined by $R(u)(v)=vu$.
Then $$\|u^*: Y^* \to X^*\|_{mb}=\|R(u):MB_0(Y,\C) \to MB_0(X,\C)\|,$$
where $MB_0(X,\C)$ denotes the closure in $MB(X,\C)$ of the set of finite rank
maps.
\end{rem}

\begin{rem} In Kirchberg's reformulation, the Connes embedding problem
asks whether the min and max norms are equivalent
on $\C \otimes \C$. If it was the case then,
for any $L= \ell_\infty(I,\C) $,  the max norm on $L\otimes \C$
would be also equal to the min norm, so that 
we would have always $X^*\otimes_{\max} \C= X^*\otimes_{\min} \C$, and hence the ``dual" embeddings $X^*\subset L_{X\subset A}$ 
 and $Y^*\subset L_{Y\subset D}$ would give us
 $\|u^*: Y^* \to X^*\|_{mb}= \|u^*: Y^* \to X^*\|_{cb}$.
        \end{rem}

\section{ Maximally bounded Schur mutipliers }\label{schur}
                       
                       Let $G$ be a discrete group.
                       We denote by $\lambda_G$ and $\rho_G$  (and sometimes simply by  $\lambda$ and $\rho$ when the choice of $G$ is clear) 
                       respectively the left and right regular representations of $G$ 
                       on $\ell_2(G)$.
                        We denote by $L(G)$ the von Neumann algebra generated by
                  $\lambda_G$, i.e. $L(G)=\lambda_G(G)''\subset B(\ell_2(G) )$. The $C^*$-algebra generated by a unitary representation $\pi$ is usually denoted by $C^*_\pi(G)$, for instance
                       $C^*_\lambda(G) \subset B(\ell_2(G))$
                       and $C^*_\rho(G) \subset B(\ell_2(G))$.
               When $\pi_u$ is the universal representation of $G$,
               we set $C^*(G)=C^*_{\pi_u}(G)$.
               We denote by  $B(G)$ the space formed of the functions on $G$ that are coefficients of unitary representations, i.e. $\phi\in B(G)$ if there is a unitary representation $\pi: G \to B(H_\pi)$ and $\xi,\eta\in H_\pi$ such that
               $\phi(.)=\langle \xi, \pi(.)\eta\rangle$. We set
               $$\|\phi\|_{B(G)} =\inf\{\|\xi\|\|\eta\| \}$$
               where the inf runs over all $\pi$ and all possible pairs $\xi,\eta\in H_\pi$. As is well known   $B(G)$ can be identified with the dual of
               $C^*(G)$. When $\phi: G \to \CC$ is finitely supported
               we may use $\pi=\lambda_G$ to compute its norm in
               $B(G)$ so that
               \begin{equation}\label{eq1}
               \|\phi\|_{B(G)}=\inf\{\|f\|_2 \|g\|_2\mid \phi= f*g\}
               \end{equation}
               where the inf runs over all factorizations of $\phi$ of the form
               $\phi= f*g$ with $f,g\in \ell_2(G)$.
               
               We denote by $\CC[G]$ the group algebra of $G$, with canonical basis $(\d_t)_{t\in G}$.
               By a (Schur or Herz-Schur) multiplier one usually means
               a transformation $M_\phi: \CC[G] \to \CC[G]$ associated to a function $\phi: G \to \CC$ that takes $\d_t$ to $\phi(t) \d_t$. We have a natural map $\CC[G] \to C^*_{\pi}(G)$ taking $\d_t$ to $\pi(t)$, and 
               $C^*_{\pi}(G)$ is the closed span of the range of the latter map.                           
                  
                  Let $X\subset A$ be an o.s. containing  $\CC[G]$
                  naturally as a dense subset, and let $Y\subset B$ be another o.s.  containing $\CC[G]$ ($A,B$ being $C^*$-algebras).
                  We say that $M_\phi$ is bounded (resp. c.b., rresp. m.b.) from $X$ to $Y$
                  if it extends to a bounded (resp. c.b., rresp. m.b.) linear map from $X$ to $Y$. In that case, we will denote (albeit abusively) by $M_\phi:X \to Y$ the resulting extended map.
                  For instance the c.b. multipliers
                  from $C^*_\lambda(G)$ to itself are characterized by
                  a celebrated theorem due to Bo$\dot{z}$ejko and Fendler \cite{[BF]}.
                  The m.b. analogue is as follows:
                  
                  \begin{pro}\label{pp1}  A multiplier  $M_\phi: \CC[G] \to \CC[G]$ 
                     is m.b.
                  on $C^*_\lambda(G)$ if and only if  the function $\phi: G \to \CC$
               is in the space $B(G)$
               and $\|M_\phi: C^*_\lambda(G) \to C^*_\lambda(G)\|_{mb}=\|\phi\|_{B(G)}$.
                  \end{pro}

                  More generally we have:

                  \begin{thm}\label{Tmul}  The following properties of a multiplier  $M_\phi: \CC[G] \to \CC[G]$ 
                  are equivalent.
                  \begin{itemize}
                  \item[{\rm (i)} ] $M_\phi$ is m.b. from $C^* (G)$ to $C^*_\lambda(G)$,
                  \item[{\rm (ii)} ] $M_\phi$ is m.b. from $C^* (G)$ to $L(G)$,
                   \item[{\rm (iii)} ] $M_\phi$ is m.b. from $C^* (G)$ to $C^* (G)$,
                   \item[{\rm (iv)} ] $\phi \in B(G)$.
                   \end{itemize}
                   Moreover in all cases the m.b. norm is equal to $\|\phi\|_{B(G)}$.
\end{thm}
                  \begin{proof} 
                  Note that 
                  the natural map $C^*(G) \to L(G)$ (resp. $C^*(G) \to C^*_\lambda(G)$) is a $*$-homomorphism
                  and hence has m.b. norm $=1$. Thus by \eqref{e21} 
                  $\|M_\phi: C^*(G) \to L(G)\|_{mb}\le
                  \|M_\phi: C^*(G) \to C^*(G)\|_{mb} 
                   $, and similarly with $C^*_\lambda$ in place of $L$.
                   We claim $  \|M_\phi: C^*(G) \to C^*(G)\|_{mb} \le \|\phi\|_{B(G)}$
                   from which follows
                   $\|M_\phi: C^*(G) \to L(G)\|_{mb}\le \|\phi\|_{B(G)}$ (and similarly with $C^*_\lambda$ in place of $L$).
                   To prove the claim, assume $\phi(.)=\langle \xi, \pi(.) \eta\rangle$.
                   Note that 
                   $t\mapsto \pi(t) \otimes \pi_u(t)$ defines a $*$-homomorphism
                    $\nu: C^*(G) \to C_\pi^*(G)\otimes_{\max} C^*(G)$
                   and the linear form $f(T)=\langle \xi, T \eta\rangle$
                   has norm (or m.b. norm) $\le \|\xi\|\|\eta\|$ from $C_\pi^*(G)$ to $\CC$. 
                   We have $[M_\phi: C^*(G) \to C^*(G)]= (f\otimes Id) \nu$ and hence 
                   (see \eqref{comp}) $ \|M_\phi: C^*(G) \to C^*(G)\|_{mb} \le \|\xi\|\|\eta\|$.
                   Taking the inf over all $\xi,\eta$
                   the claim follows.

                  Conversely 
                  let $x\in \CC[G]$. Assume $\|M_\phi: C^*(G) \to L(G)\|_{mb}
                  \le 1 $.
                  We have by universality
                  $$\|x\|_{C^*(G)}=\|\sum x(t) \pi_u(t)\|\ge \|\sum x(t) \pi_u(t)\otimes \rho(t)\|_{C^*(G)\otimes_{\max} C_\rho^*(G)}$$
                  and hence                    $$ \ge  \|\sum x(t) \phi(t) \lambda(t)\otimes \rho(t)\|_{L(G)\otimes_{\max} C_\rho^*(G)} \ge  \|\sum x(t) \phi(t) \lambda(t) \rho(t)\|_{B(\ell_2(G))}$$
  and since $\langle \d_e, \lambda(t) \rho(t) \d_e\rangle=1$ for all $t$ this is
  $\ge |\sum x(t) \phi(t)|$, which means
  $$\|x\|_{C^*(G)} \ge |\sum x(t) \phi(t)|.$$
  This shows $\|\phi\|_{C^*(G)^*} \le1$, and 
         since $B(G)=C^*(G)^*$ isometrically  $\|\phi\|_{B(G)} \le1$.  By homogeneity
         we obtain $\|\phi\|_{B(G)} \le \|M_\phi: C^*(G) \to L(G)\|_{mb}$. 
         By   \eqref{e21},  $\|M_\phi: C^*(G) \to L(G)\|_{mb}$ is less or equal to the other norms in (i) and (iii).
         The proof is then completed easily.
              \end{proof}

          \begin{proof}[Proof of Proposition \ref{pp1}] The inequality $\| M_\phi: C_\lambda^* (G) \to C_\lambda^* (G)  \|_{mb} \le \|\phi\|_{B(G)}$ can be proved   using the Fell absorption principle by which
          for any $\pi$ the map $\lambda(t) \mapsto [\lambda \otimes \pi ] (t)$ extends to
          a (faithful) $*$-homomorphism on $C_\lambda^* (G)$. Then using the functional $f(T)=\langle \xi, T \eta\rangle$ and reasoning as above
          we obtain the   inequality in question.\\
                   The converse follows from Theorem \ref{Tmul} and \eqref{comp}
                   since the natural morphism $\pi_1: C^* (G) \to C_\lambda^* (G)$
                   has m.b. norm $=1$, and
                   $[M_\phi: C^*(G) \to C_\lambda^* (G)]= [M_\phi: C_\lambda^* (G) \to C_\lambda^* (G)]\pi_1,$
                   we have $\|\phi\|_{B(G)} =\|  M_\phi: C^*(G) \to C_\lambda^* (G)  \|_{mb} \le \| M_\phi: C_\lambda^* (G) \to C_\lambda^* (G)  \|_{mb}.$
                      \end{proof} 
                        \begin{rem}
                      Let $\Lambda\subset G$ be an arbitrary subset.
                      Let $\phi: \Lambda \to \CC$. 
                      Let $C^*(G)_\Lambda \subset C^*(G)$ 
                      and
                      ${C^*_\lambda(G)}_\Lambda \subset C_\lambda^*(G)$
                       be the corresponding closed subspaces generated by $\Lambda$. 
                      Repeating the preceding argument and 
                      using the Hahn-Banach theorem
                       shows that
                      $$\|M_\phi : C^*(G)_\Lambda \to L(G)\|_{mb} \le 1
                      \text{  (or equivalently  }   \|M_\phi : {C^*_\lambda(G)}_\Lambda \to {C^*_\lambda(G)}_\Lambda\|_{mb} \le 1)  $$
                      if and only if $\phi $ admits an extension
                      $\tilde \phi : G \to \CC$ with $\| \tilde \phi \|_{B(G)}\le 1$.
    \end{rem}

                      Some of the preceding facts   are similar to observations made in \cite[Prop. 2.4 ]{147}
                       about decomposable multipliers.
                      
                      \begin{rem}
                      In \cite[Cor. 2.6]{JLM} Junge and Le Merdy prove 
                      that $\|u\|_{dec}=\|u\|_{mb}$  for any {\it finite rank} map
                      $u: A \to B$ between $C^*$-algebras.
                     We will say that a $C^*$-algebra $A$ has the 
                     $c$-MBAP if the identity is the pointwise limit of a net
                     of finite rank maps with m.b. norms at most $c$.
                     By  their characterization of {\it finite rank} maps with 
                     m.b. norm at most $1$ on $A$ (see \cite[Th. 2.1]{JLM}) 
                    this holds if and only if $A$ is nuclear. This extends Lance's classical 
                     theorem that the CPAP  is equivalent to nuclearity.
                  It follows that the MBAP of $C_\lambda^* (G)$ (or of $C ^* (G)$) is equivalent to the amenability of $G$, in sharp contrast to the CBAP. 
    \end{rem}

        \section{The mb-norm and strong similarity}\label{s2}

       Following partial results due to Hadwin and Wittstock,
    in \cite{Ha}  Haagerup proved that     
     a unital homomorphism
    $u: C \to B(H)$ on a $C^*$-algebra $C$ is
    c.b. if  and only if  it is similar to a $*$-homomorphism and moreover that its cb-norm  is
     equal to the (attained) infimum of $  \|S\|\|S^{-1}\|$
 over all invertibles $S \in B(H)$ such that $S u(.) S^{-1}$ is   a $*$-homomorphism.
 (see   \cite[Cor. 4.4, p. 79]{95} for more on this).

    The following identity is the analogous fact for
    the mb-norm of $u$:
    
    \begin{thm}\label{s1} Let $u: C \to {  N}$ be a unital homomorphism from a $C^*$-algebra $C$ to a von Neumann algebra ${  N}\subset B(H)$.
    Then  $u$ is m.b. if  and only if  $u$ is strongly similar to a $*$-homomorphism, and 
    $$\|u\|_{mb} =\inf\{\|S\|\| S^{-1}  \| \}$$
    where the infimum runs over all invertibles $S \in {  N}$
    such that $S u(.) S^{-1}$ is   a $*$-homomorphism.
    \end{thm}
      We first prove an expanded version, essentially
        observed in \cite[(2) p. 81]{toho}. See also \cite[\S 6.3]{157}.
     \begin{pro} \label{p1}
        Let $u: C \to N  $  be as in Theorem \ref{s1}.        Let $v: N' \to B(H)$ be the inclusion map.
        Then $u$ is strongly similar to a $*$-homomorphism
        if  and only if  the map $v . u : N' \otimes  C\to B(H) $ (defined by
        $v . u (a\otimes b)= v(a)u(b) $ extends to
        a c.b. map $w$ from $N' \otimes_{\max} C$ to $B(H)$.
        Moreover
         \begin{equation}  \label{e9}
         \| w:  N'\otimes_{\max} C \to B(H)\|_{cb} =\inf\{\|S\|\| S^{-1}  \| \}\end{equation}
        where the (attained) inf runs over all invertible  $S\in N$
        such that  $S u(.) S^{-1}$ is   a $*$-homomorphism.\\
        Lastly, if $N'$ admits a cyclic vector in $H$ we have $\|w\|=\|w\|_{cb}$.
\end{pro}
          \begin{proof}  Assume $w$ c.b. Clearly $w$ is still a unital homomorphism,
          and hence if it is c.b. it must be similar to a $*$-homomorphism
          whence the existence of an invertible $S$ on $H$ (with
          $\|S\|\| S^{-1}  \| =\|w\|_{cb} $)
          such that $x\mapsto S w(x) S^{-1} $ is a $*$-homomorphism.
          But restricting this to $N' \otimes 1$ we find that 
          $N' \ni y \mapsto  S y S^{-1}$ is  a $*$-homomorphism. A fortiori
          it is self-adjoint, 
          which implies by a simple calculation 
          ($S y S^{-1}   = (S y^* S^{-1})^* $ and hence $S^*S y=yS^*S$)
          that $|S| \in N''=N$.
          Now $S=U|S|$ with $U$ unitary and the map
          $C\ni x \mapsto  U (|S| u(x) |S|^{-1}) U^{-1}$
          is  a $*$-homomorphism.
          Since $U$ is unitary $C\ni x \mapsto    (|S| u(x) |S|^{-1})  $
           is  a $*$-homomorphism, proving the strong similarity of $u$.
           
           Conversely, if $S\in N$ is such that $C\ni x \mapsto    (S u(x) S^{-1})  $
           is  a $*$-homomorphism denoted by $u_S: C \to N$,
           then obviously $v.u_S $ is a (completely contractive)  
           $*$-homomorphism from $ N'\otimes_{\max} C$ to $B(H)$.
           Then  $a\otimes b \mapsto S^{-1} v.u_S (a\otimes b)   S $
           has cb-norm at most $\|S\|  \| S^{-1}  \|$ on $ N'\otimes_{\max} C$ and since $S$ commutes with $v(b)$, we have
           $S^{-1} v.u_S (a\otimes b)   S = v.u(a\otimes b) $. Thus we 
           conclude that $$\| v.u: N'  \otimes_{\max} C\to B(H)\|_{cb} \le \inf\|S\|  \| S^{-1}  \|.$$
           The converse inequality was proved in the first part. There we also proved that the inf is actually a minimum.\\
             The last assertion follows from  \cite{smith} (see the next remark, already stated in    \cite[Cor. 6.23]{154}).
    \end{proof}
    \begin{rem}\label{cyc} Let $u: E \to N$ be a linear map from
    an operator subspace $E\subset C$ to a von Neumann algebra $N \subset B(H)$.
    Let $\a$ be a $C^*$-norm on $N' \otimes_{alg} C $. 
    Assume that  the map $w: N' \otimes_{alg} E \to B(H)$ defined as above by $w(y\otimes x)= y u(x)$ ($y\in N', x\in E$) extends to a bounded map
    on $\ovl{N' \otimes_{alg} E}^\a$. Then if $N'$ admits a cyclic vector, $w$ is c.b. and
 $$\|w\|_{cb}=\|w\|.$$
    Indeed, since $w$ is an $N'$-bimodule map this follows from
    Roger Smith's result in  \cite{smith}.
    \end{rem}

    \begin{proof}[Proof of Theorem \ref{s1}] By \eqref{e10}
    with $D={N}'$ and $E=C$  
    and by \eqref{e9}  there is $S\in N$ with
   $ \|S\|\| S^{-1}  \|\le \|u\|_{mb}$
    such that $u_S(.)=S u(.) S^{-1}$ is   a $*$-homomorphism into $N$.
    Conversely, since a  $*$-homomorphism has unit mb-norm,
    and left or right multiplication by an element  $x\in N$
    has mb-norm at most $\|x\|$, the fact that
     $u(.)=S^{-1}u_S(.)S$ with   $S\in N$ gives us
    $\|u\|_{mb} \le  \|S\|\| S^{-1}\| $.
     \end{proof}
     
     Haagerup's result was extended by Paulsen \cite{Pa}
     to the non self-adjoint case to study dilation theory following
     a  tradition initiated by     Sz.-Nagy  in his famous paper \cite{Nagy}
      (see \cite[p. 7 and p. 47]{Pau} or \cite[p. 15 and chap. 4]{95} for more details).
         In the situation of the next statement, he proved that a unital homomorphism $u: C \to B(H)$
         is similar to a c.c. one  if and only if it is c.b.
          The analogue   for the mb-norm
     is as follows.

      \begin{thm}\label{pam} Let $A$ be a unital $C^*$-algebra. Let $u: C \to {  N}$ be a unital homomorphism from a  unital operator subalgebra $C\subset A$  to a von Neumann algebra ${  N}\subset B(H)$.
  If $\|u\|_{mb}<\infty$ there is an invertible $S \in {  N}$
  with $\|S\|\| S^{-1}  \|= \|u\|_{mb}   $
       such that $u_S(.)=S u(.) S^{-1}$ is   a  unital   homomorphism with  mb-norm $=1$.
       More precisely, $u$ is strongly similar to a
       homomorphism  with mb-norm $=1$  if and only if it is   m.b. 
      and  we have
         $$\|u\|_{mb} =\inf\{\|S\|\| S^{-1}  \| \mid S\in N \ \ \|u_S\|_{mb}=1\}.$$
  \end{thm}
      
      \begin{proof} Let $w:  N'\otimes_{\max} C \to B(H)$ be as in the preceding proof, so that by \eqref{e10'}
      $\|w\|_{cb} \le \|w\|_{mb} \le \|u\|_{mb}$. By Paulsen's result
      there is $S\in GL(H)$ such that
      $\|w_S\|_{cb} =1$ with $\|S\| \|S^{-1}\|  \le \|w\|_{cb}$. 
       Restricting to $N'\otimes 1$ as above,
      we see that we may assume $|S|\in N$  and of course
      $w_{|S|}$ is still unital  and completely contractive (c.c.).
      We claim  that $\|u_{|S|}\|_{mb} =1$. 
      Indeed, $w_{|S|}$  extends to a c.c. unital and hence u.c.p. map 
      $\phi : N'\otimes_{\max} A \to B(H)$ (by the injectivity of $B(H)$).
      This c.p. map $\phi$
       must be $N'$-bimodular since $N'\otimes 1$ is in its  multiplicative domain (see e.g. \cite[p. 107]{154}). Let $f: A \to B(H)$ be the map
      such that $\phi(x\otimes a)= x f(a)$ (i.e. $f(a)=\phi(1\otimes a)$).
      The bimodularity of $\phi$ implies that $f(a)\in (N')'=N$.
      Since $f: A \to N$ is a u.c.p. map between
      $C^*$-algebras, we have $\|f\|_{mb}=1$ (see Proposition  \ref{L1})
      and  since $f_{|C}=u_{|S|}$  we find $\|u_{|S|}\|_{mb} =1$,
      proving the claim.
       \\ Conversely, if $S\in N$ is invertible such that $\|u_S\|_{mb}= 1$ 
      then as above $\|u\|_{mb} \le \|S\|\| S^{-1}  \| \|u_S\|_{mb}=\|S\|\| S^{-1}  \|$.
      This completes the proof.
      \end{proof}
      
      The most classical application   is Paulsen's criterion
      for a single operator $T\in B(H)$: the operator $T$ is similar to a contraction if and only if the functional calculus map $P\to P(T)$ ($P$ a polynomial)
      defines a completely contractive map $\rho_T : {\cl A} \to B(H)$  from the disc algebra 
      ${\cl A} \subset C(\partial D)$ (i.e. the closure of the set of polynomials)
       to $B(H)$.
       Note that $\|\rho_T\|_{mb}=\|\rho_T\|_{cb}$ since $C(\partial D)$ is nuclear and
       $B(H)$ has the WEP.
       We may write:
       $$\|T\|\le 1 \Leftrightarrow \|\rho_T\|=1 \Leftrightarrow \|\rho_T\|_{cb}=1 \Leftrightarrow \|\rho_T\|_{mb}=1 .$$

      The analogue of Paulsen's criterion is as follows. Note  
      we recover the original one with $N=B(H)$.
      \begin{cor} Let $T\in B(H)$  
      and let $N$ 
      be a von Neumann algebra containing $T$.
  The map $\rho_T: {\cl A} \to N$
      is m.b. if  and only if  there is $S\in N$ invertible  such that $\|STS^{-1}\| \le 1$, and
      $$\| \rho_T\|_{MB({\cl A},N)}=\inf\{\| S\| \| S^{-1}\|  \}$$
      where the inf (which is attained) runs over all such $S$.
       \end{cor}
      \begin{rem} The preceding proof actually shows that
      $\|u\|_{mb}=\|w\|_{cb}$, and if $N' \subset B(H)$ is cyclic $\|u\|_{mb}=\|w\|$
      by Remark \ref{cyc}.
      \end{rem}
       
         \def\t{\theta}
         
    Let $\pi: A \to B(H)$ be a unital $*$-homomorphism and
  $\d: A \to B(H)$ be a $\pi$-derivation (meaning $\d(xy)=\pi(x) \d(y)+\d(x)\pi(y)$ for all $x,y\in A$). 
        For any $T$ in $ B(H) $ we denote
      $$\d^T(x)=[\pi(x),T].$$
      Most of the time, it suffices to consider derivations from a unital subalgebra of $  B(H)$. Indeed,
       since the derivation $\d$ vanishes on $\ker(\pi)$ it factorizes and  defines
      a derivation from $\pi(A)$ to $B(H)$ relative to the inclusion $\pi(A)\subset B(H)$.
      
By Christensen's theorem (see \cite[p. 80]{95}) the derivation $\d$ is c.b. 
      if and only if it is inner, i.e. $\d=\d^T$ for some $T\in B(H)$, and 
      moreover,  since the c.b. norms of $\d^T$ on $\pi(A)$  and   $\pi(A)''$
      are equal, 
      Christensen's formula (see \cite[p. 96]{95}) gives us
      $$\|\d\|_{cb} =2\inf \{\|\t\| \mid \t\in B(H),  \d=\d^\t \} =d(T, \pi(A)').$$
      In addition, the preceding infimum is attained on some $\t\in B(H)$.
      
     The case of m.b. derivations leads to a  generalization of the latter formula:

     \begin{thm}\label{t5.6} Let $M\subset B(H)$ be a von Neumann algebra,
         $\pi: A \to M$   a unital $*$-homomorphism on a (unital) $C^*$-algebra $A $, and
  $\d: A \to M$   a $\pi$-derivation. 
  Then $\|\d\|_{mb} <\infty $ if  and only if  there is $\t \in  M$ such that
  $\d=\d^\t$ and we have
  \begin{equation}\label{chr}
  \|\d\|_{mb} =2\inf\{ \|T\| \mid T\in  M, \ \d=\d^T\} =d(\t,\pi(A)'), \end{equation}
  where again the infimum is attained.
  Moreover
  $\|\d\|_{mb}$ coincides with the cb-norm of the 
  mapping $D: M' \otimes_{\max} A \to B(H)$ defined by $D(x\otimes a)=\pi(x)\d(a)$, and
  $\|\d\|_{mb}=\| D\|$ if
  $M'$ is cyclic.
 
     \end{thm}
      \begin{proof}  Assume $\|\d\|_{mb} <\infty $.
      Then by \eqref{e10} we have
      $\|Id \otimes \d :  M' \otimes_{\max} A \to M' \otimes_{\max} M\|_{cb} \le \|\d\|_{mb}$
      and hence $\|D\|_{cb} \le \|\d\|_{mb}$.
      Let $\Pi: M' \otimes_{\max} A \to B(H)$ be the continuous $*$-homomorphism
      defined by $\Pi(x\otimes a)=x\pi(a) $. It is easy to check that
      $D$ is a $\Pi$-derivation. By Christensen's theorem  
      there is $T\in B(H)$ with $\|T\|= (1/2)\|D\|_{cb}$ such that $D(z)=[\Pi(z), T]$ for any $z\in
      M' \otimes_{\max} A$. 
      Using $z\in M' \otimes 1$ (for which  $D(z)=0$) we find $T\in M''=M$ and using
      $z\in 1 \otimes A$ we find $\d(a)=[\pi(a),T]$, i.e. $\d=\d^T$.
      This proves $2\inf\{\|\t\| \mid \t\in M, \d=\d^\t\} \le 2\|T\|= \|D\|_{cb} \le  \|\d\|_{mb}$.
      Conversely, if $\t\in M $ and $ \d=\d^\t$ we clearly have $\|\d\|_{mb} \le 2 \|\t\|$,
      and hence $\|\d\|_{mb} \le 2 \inf\{\|\t\| \mid \t\in M, \d=\d^\t\}$. 
      This proves  $\|D\|_{cb} =  \|\d\|_{mb}$
      and the first equality in \eqref{chr}, the second one follows since $\d^T=\d^\t \Leftrightarrow T\in \t +\pi(A)'$.
      It also proves in passing that $\inf\{\|\t\| \mid \t\in M, \d=\d^\t\} $ is attained on $T$.
      The cyclic case follows by Remark \ref{cyc}.
     \end{proof}
     
     \def\t{\theta}
     
     The extension property in the next statement will play a special role in
     the proof of Theorem \ref{T10}.  
     \begin{cor}\label{ptruc} Assume that $ A \subset M$ and that $\pi$ is the inclusion map.
     Then any m.b. derivation $\d: A\to M$ extends to an m.b. derivation
     $\tilde \d: M \to M$ with  $\|\tilde \d\|_{cb} \le  \|\tilde \d\|_{mb}=  \|  \d\|_{mb}$.
     \end{cor}
     \begin{proof} We just set $\tilde \d(x)=[x,T]$ with  $T$ attaining the infimum in \eqref{chr}.
     \end{proof}

\section{A few facts  on row and column operator spaces}\label{bg1}

We need to recall a few facts from o.s. theory in preparation for the
proof of the main result of the next section.

The direct sum of two operator spaces $E_1\subset A_1$ and  
$E_2\subset A_2$ is the subspace $E_1\oplus E_2 \subset A_1\oplus A_2$, the latter being the  direct sum in the $\ell_\infty$ sense of the $C^*$-algebras.

Given two Hilbert spaces $H,K$ we may embed $B(H,K)$ in 
$B(H\oplus K)$ in a natural way
thus equipping $B(H,K)$ with an o.s. structure.
This applies in particular to $B(\CC, H)$ (the ``column" structure) 
and to $B(H^*, \CC)$ (the ``row" structure). The resulting
o.s. are denoted respectively by $H_c$ and $H_r$. In particular, we have
$R=(\ell_2)_r$ and $C=(\ell_2)_c$ denoting the  Hilbert spaces 
formed of row and column matrices in $B(\ell_2)$.
Let $t=\sum b_j \otimes x_j\in B \otimes H$ (here $B$ can be any $C^*$-algebra).
 We have
 $$\| t\|_{B\otimes_{\min} H_c} =\|\sum b_i^*b_j \langle x_i, x_j\rangle \|^{1/2}_B
\text{  and  } 
\| t\|_{B\otimes_{\min} H_r} =\|\sum b_ib^*_j \langle x_j, x_i\rangle \|^{1/2}_B.$$

Associated to a $C^*$-algebra $B$ are its opposite $B^{op}$
defined as the algebra obtained from $B$ after reversing the order of the product, and its complex conjugate $\ovl B$, which is naturally anti-isomorphic to $B$. Note that the map $x\mapsto x^*$ defines
a bona fide isomorphism from $B^{op}$ to $\ovl B$, so that they can be thought of as the same algebra.

There is a special self-dual o.s. structure on $H$ called the ``OH" one,
which produces the o.s. $H_{oh}$. We have
$$\| t\|_{B\otimes_{\min} H_{oh}} =\|\sum \ovl{ b_i}\otimes  b_j \langle x_i, x_j\rangle \|^{1/2}_{\ovl B \otimes_{\min} B} .$$

Let $(M, \tau)$ be a von Neumann algebra equipped with a standard
normalized  trace $\tau$. 
The preceding applies to the Hilbert space $L_2(\tau)$, whence
the o.s.  $L_2(\tau)_c$, $L_2(\tau)_r$ and $L_2(\tau)_{oh}$.

\begin{lem}\label{rcoh}
The inclusions $M \to L_2(\tau)_c$,
 $M \to L_2(\tau)_r$ and $M \to L_2(\tau)_{oh}$
  are all
 completely contractive.
\end{lem}
\begin{proof} 
Let $t=\sum b_j \otimes x_j\in B \otimes M$.
We have  $\|t\|^2_{\min}=\| t^* t\|_{\min}=
\| \sum b_i^* b_j \otimes x_i^* x_j\|_{\min}$, and
since the linear form $x\mapsto \tau(x)$ is completely
contractive from $M$ to $\CC$, the latter is
$$\ge \| \sum b_i^* b_j  \tau(x_i^* x_j)\|_B= \| \sum b_i^* b_j  \langle x_i, x_j\rangle\|_B. $$
 This settles the case of $M \to L_2(\tau)_c$. The case of $r$ is similar.
 The case of $oh$ follows by complex interpolation. Indeed, by \cite[p.137]{P4}
 when $B=M_n$ we have $(B\otimes_{\min } H_c, B\otimes_{\min } H_r)_{1/2}= B\otimes_{\min } H_{oh}$. This implies that denoting by
 $j: M \to L_2(\tau)$ the inclusion, the map
  $Id_B \otimes j : B\otimes_{\min } M \to B\otimes_{\min }  L_2(\tau)_{oh}$
  has norm $\le 1$, since it does so for both cases $c$ and $r$.
\end{proof}
\begin{rem}\label{RC}

Consider the ``diagonal" embedding $x\mapsto (x,x)$ of $H$ into
the o.s. $H_r \oplus H_c$.
We denote by $H_{rc} \subset H_r \oplus H_c$
the o.s. obtained by equipping $H$ with the o.s. structure induced
by $H_r \oplus H_c$ on this diagonal subspace.
For any $t\in B \otimes H$ we have
\begin{equation}\label{e122}\|t\|_{B \otimes_{\min} H_{rc}} =\max\{ \|t\|_{B \otimes_{\min} H_{r}} ,
\|t\|_{B \otimes_{\min} H_{c}} \}. \end{equation}
The space $H_{rc}$ is often (abusively)  denoted $H_r \cap H_c$ to reflect \eqref{e122}.
Similarly, one sets 
$R\cap C=(\ell_2)_{rc}$.

\end{rem}
\begin{lem}\label{rcoh1}
The inclusion $M \to L_2(\tau)_{rc}$ is  completely contractive. Moreover,
for any $H$ the identity on $H$ defines a completely contractive map $H_{rc} \to H_{oh}$. In particular,
 the identity on $L_2(\tau)$ defines a completely contractive map $L_2(\tau)_{rc} \to L_2(\tau)_{oh}$.
\end{lem}
\begin{proof} 
By classical interpolation theory (see e.g. \cite{BeL}) for any interpolation pair $(A_0 , A_1)$ we have   $\|A_0 \cap A_1\to (A_0 , A_1)_{1/2} \|\le 1$, and hence for any $B$  $$\|(B\otimes_{\min } H_r) \cap (B\otimes_{\min } H_c) \to 
(B\otimes_{\min } H_r, B\otimes_{\min } H_c)_{1/2} \|\le 1.$$
Since    $(B\otimes_{\min } H_r, B\otimes_{\min } H_c)_{1/2}= B\otimes_{\min } H_{oh}$ when $B=M_n$,  the result follows by \eqref{e122}.
\end{proof}
\begin{rem}\label{RC1}
Let $(e_j)$ be an orthonormal basis of $H$.
Then for any finitely supported $j\mapsto b_j \in B$
we have
$\|\sum b_j \otimes e_j \|_{B \otimes_{\min} H_{r}} =
\|\sum b_j b_j^*\|^{1/2}$,
$\|\sum b_j \otimes e_j \|_{B \otimes_{\min} H_{c}} =
\|\sum b^*_j b_j\|^{1/2}$ and
$$\|\sum b_j \otimes e_j \|_{B \otimes_{\min} H_{rc}}
= \max\{\|\sum b_j b_j^*\|^{1/2} ,
\|\sum b^*_j b_j\|^{1/2}\}. $$
Moreover
$$\|\sum b_j \otimes e_j \|_{B \otimes_{\min} H_{oh}}
= \|\sum \ovl{b_j} \otimes b_j \|_{\ovl B \otimes_{\min} B }^{1/2} .$$
\end{rem}
\begin{rem}\label{homo} An o.s.  $E$ is called homogeneous if $\|u\|_{cb}=\|u\|$
for any $u: E \to E$. 
The spaces $H_r,H_c,H_{oh}$ and $H_{rc}$ are all examples of homogeneous o.s. 
This implies that   any operator subspace $K\subset H_r$ can be identified with $K_r$ and similarly for $c, oh$ and $rc$.
\end{rem}

 \begin{rem}[On $R_n\cap C_n$]\label{rcapc}
         Consider in the direct sum $M_n \oplus M_n$
     the $n$-dimensional subspace $F_n \subset M_n \oplus M_n$ spanned by the family
   $\{e_{1j} \oplus  e_{j1} \}$.  Here we view $M_n \oplus M_n \subset M_{2n}$
   by identifying the elements of $M_n \oplus M_n$ with block diagonal $2n \times 2n$-matrices. This space is often denoted (abusively)  by
         $R_n\cap C_n$   in the context of interpolation between pairs of o.s. It plays an important role in connection
   with free probability questions (see \cite{HP2}).
   \end{rem}

For all these facts we refer to \cite{ER,P4,Pau}.

      \section{One more characterization of injectivity}\label{inj}
           \def\b{\beta}

      Given a f.d. o.s. $F$ and a $C^*$-algebra (or  an o.s.) $B$, we denote
      \begin{equation}  \label{e128} C(F,B)=\inf\{ \|w\|_{cb}  \|v\|_{cb}\} \end{equation} 
     where the inf runs over all factorization of $Id_{F}$ through $B$ of the form
  $  F {\buildrel {w}\over \rightarrow} B {\buildrel {v}\over \rightarrow}  F $.
     
 In particular, 
 concerning the space $F_n$ (or ``$R_n\cap C_n$") defined in
   Remark \ref{rcapc},
    the following result was proved in  \cite{HP2}:

      \begin{thm}\label{hp2}
   For any  $C^*$-algebra $B$ with the WEP we have
         $$C(F_n,B)\ge  (1+\sqrt{n})/2 ,$$
         and equality holds when $B=B(H)$ or $B=M_{2n}$.
   \end{thm}
   
 This is proved in \cite{HP2} for $B$ either nuclear or injective.
   If $B$ has WEP, the inclusion $B\to B^{**}$ factors through $B(H)$
   via complete contractions. Since $v: B\to F_n$ extends to $v^{**}:B^{**} \to F_n$, a factorization through $B(H)$ appears inside the one through $B$ and  the preceding statement follows immediately.
   
In particular, for any injective von Neumann algebra $M$.\
there is $c>0$ such that we have $C(F_n,M)\ge  c\sqrt{n}  $ for all $n$. 
 In Theorem \ref{t3.1}, we will prove
that conversely if for some $c>0$ we have $C(F_n,M)\ge  c\sqrt{n}  $  for all $n$, then $M$ is injective.

This converse was already conjectured 
     around the time of  \cite{HP2}. In informal discussions with the first author,
            partial unpublished (and unwritten) results were obtained by
            Franz Lehner at the time of his 1997 PhD thesis, these covered
           the finite case when $M= L_\infty([0,1]) \bar\otimes N$.
       More recently, for $M=L(G)$  von Neumann algebra  of a discrete group $G$, it is easy to see that
               $C (F_n, M) \ge c \sqrt n$ implies $\|\sum \lambda_G(s_j) \|\ge  c^2  n$ 
               for any $n$-tuple $(s_j)$ of ({\it distinct}) elements of $G$, and if this holds for all $n$  it follows that $G$ is amenable by Andreas Thom's \cite{Th}.

          \def\M{{\cal M}}

            \begin{thm}\label{t3.1}     A von Neumann algebra   $M$
     is injective if and only if  there is $c>0$ such that $C (F_n, M) \ge c \sqrt n$ for all $n$.                
             
           \end{thm}

         To prove  the ``if part"  we will use Haagerup's criterion of injectivity  in \cite[Lemma 2.3]{Ha1132}  which is originally due to Connes \cite{Co} for factors.
           The main point for now is as follows.
              \begin{pro}\label{p2} Let $H=\ell_2$.
           Let $M$ be a finite  von Neumann algebra  with normalized trace $\tau$.
            Let $\a: H_{rc} \to M$ be any c.b. map.
            Let $k: M \to L_2(\tau)_{r c}$ be the
            inclusion as in Lemma \ref{rcoh1}.
         Let $T = k\a : H_{rc} \to L_2(\tau)_{r  c}$.
         Then 
            \begin{equation}  \label{e16} \forall n\ge 1\quad   a_n(T) C (F_n, M) \le \|T\|_{cb} \le \|\a\|_{cb},\end{equation}
            where $$a_n(T)=\inf\{\|T-T'\| \mid rk(T')<n\}.$$
              \end{pro}
                       
              \begin{proof} Recall
                            $\|T\|=a_1(T) \ge \cdots\ge a_n(T)$ and $a_n(T)=\lambda_n(|T|).$
              The polar factorization shows (as is classical) that there is a
              subspace $F\subset H_{rc}$ of dimension $n$  such that  $\|T(x)\| \ge a_n(T) \|x\|$
              for any $x\in F$. We may assume  $  a_n(T)>0$. It follows that
              $\|{T^{-1}}_{|T(F)} \|\le a_n(T)^{-1}$.
              Let $P: L_2(M)_{r  c} \to T(F)$ be the orthogonal projection.
              Note $\|P\|_{cb} \le 1$ and $\|{T^{-1}}_{|T(F)} \|_{cb}=\|{T^{-1}}_{|T(F)} \| $ by homogeneity (see Remark \ref{homo}). We have $Id_F= {T^{-1}}_{|T(F)} P k \a_{|F} $ and $\|k\|_{cb}\le 1$ by Lemma \ref{rcoh}.
              Therefore
              $$C(F,M)\le \|{T^{-1}}_{|T(F)} \|_{cb}\|\a\|_{cb}   \le a_n(T)^{-1}\|\a\|_{cb} .$$
             By homogeneity again,   $F\simeq F_n$ completely isometrically,
             so that $C(F,M)=C(F_n,M)$ and the announced result follows.            
                \end{proof} 
                
                  \begin{cor} 
                  Assume  $C (F_n, M) \ge c \sqrt n$ for all $n$. 
                  Then for any $n$ and any $\a: H_{rc} \to M$ we have
                  $$\sup \nl_n n^{1/2} a_n(k\a)   \le c^{-1} \|\a\|_{cb}.$$
                 \end{cor}

             \def\g{\gamma}

   \begin{proof}[Proof of Theorem \ref{t3.1}] 
 To start with we claim that it suffices to prove the theorem in the finite case.
\\ Indeed, if $N\subset M$ is a von Neumann  subalgebra admitting
           a c.c. projection $P: M \to N$ we have obviously
           $$C(F_n, M) \le C(F_n, N).$$
           Therefore the assumption on $M$
           is inherited by $N$, and 
           in particular by $N=pM$ for any central projection $p$ in $M$.
           This shows that the three terms in the decomposition  $M=M_I\oplus M_{II}\oplus M_{III}$ in 
           types I, II (semifinite)  and III all inherit the assumption.
           The type I case is easy since it is always injective.
           Using the subalgebras $pM$ with $p$ finite, 
           which have a weak*-dense directed union, the semifinite case is easily reduced to the finite one.
            Then using  crossed products and Takesaki's
           duality theorem as in \cite[p. 198]{Ha1132} one can treat the type III case.
 This proves our claim that we may assume $M$ finite.

    Assume  $M$ finite and $C (F_n, M) \ge c \sqrt n$ for all $n$. \\
    We will use the traditional notation $OH_n=(\ell_2^n)_{oh}$.
 We have a natural c.c. inclusion
                 $j_n : F_n \to OH_n$
                 (see Lemma \ref{rcoh1}). Let $e_j$ (resp. $t_j$) be an orthonormal basis
                 in $F_n$ (resp. $OH_n$) with (say) $j_n(e_j)= t_j$. Let $y_j\in M$.
                 Consider the map $\g_n : OH_n \to M$
                 taking $t_j$ to $y_j$. 
                 Let $\a_n= \g_nj_n : F_n \to M$ so that $a_n(e_j)= y_j$.
                 We have
                 $$ \| \a_n\|_{cb} \le  \| \g_n\|_{cb}=  \|\sum  \bar{y_j} \otimes y_j \|_{\min}^{1/2} .$$
     Then \eqref{e16} yields the following:
                     for any $(y_j)$ in $M$ the map $\g_n: OH_n \to M$ taking   $(t_j)$ to $(y_j)$
                       satisfies
                       $$\| k \g_n \|_{S_{2,\infty} } =\sup\nl_m \sqrt m  a_m(k \g_n)  \le c'  \|\sum  \bar{y_j} \otimes y_j \|_{\min}^{1/2},$$
                       where $c'=1/c$.       
In other words the map $w_n : \ell_2^n \to L_2(\tau) $
taking $e_j$ to $y_j$ satisfies
 \begin{equation}  \label{e15}
 \| w_n \|_{S_{2,\infty} } \le c' \|\sum  \bar{y_j} \otimes y_j \|_{\min}^{1/2}. \end{equation}
   Since $w_n$ is of rank at most $n$ 
   we have for all $n>1$
   $$\|w_n\|_{S_2}\le \|w_n\|_{S_{2,\infty}} c''\sqrt{\log n}.$$
   Therefore for any $n$-tuple of unitaries $(y_j)$ in $M$
   $$\sqrt n \le c''\sqrt{\log n} \ c' \|\sum  \bar{y_j} \otimes y_j \|_{\min}^{1/2}.$$
   We  then use a well known trick.
   Let $T= \sum  \bar{y_j} \otimes y_j $. Note
   $\|T\|^{2m}=\|(T^*T)^m\|$ and 
   $(T^*T)^m$ when developed is a sum of $n^{2m}$ 
   terms of the form $\bar U\otimes U$ with $U$ unitary in $M$.

   Therefore
   $$ n^m \le (c''c')  \sqrt{\log(n^{2m}) } \|T\|^{m}$$
   and hence
   $$n \le [c''c'  \sqrt{\log(n^{2m}) } ]^{1/m} \|T\|$$
   and letting $m\to \infty$ we obtain
   $$n \le  \|T\|.$$
   This shows that  $\|\sum  \bar{y_j} \otimes y_j \|_{\min} =n$
   for any $n$-tuple of unitaries in $M$ and any $n$.
   This settles the case of factors,
   but since $pM$ satisfies the same assumption
   for any central projection $p\in M$,
   the same conclusion holds holds for $pM$.
   By Haagerup's criterion in \cite[Lemma 2.3]{Ha1132} (due to Connes \cite{Co} for factors), $M$ is injective.
   This proves the ``if part". The converse holds by   \cite[Th. 4.1]{HP2}.
 \end{proof}

 \begin{rem}[The case of $OH_n$] The preceding argument is quite similar to one previously used by Marius Junge in \cite{Ju-oh} for the operator Hilbert space $OH_n$ in place of $F_n=R_n\cap C_n$.
 The latter space $OH_n$ (for which we refer to \cite[\S 7]{P4}) is a copy of the $n$-dimension Hilbert space $\ell_2^n$
 sitting  isometrically in $B(H)$ in such a way that
   its orthonormal basis corresponds to elements $t_j\in B(H)$
 such that for any other Hilbert space $K$ and any $a_j \in B(K)$ we have the characteristic identity
 $$\|\sum a_j \otimes t_j \|_{B(K) \otimes_{\min} B(H)} =\|\sum \ovl{a_j} \otimes a_j \|^{1/2}_{\ovl{B(H)} \otimes_{\min} B(H)}.$$
  In  \cite{Ju-oh}  Junge shows that
 for any WEP $C^*$-algebra $A$ we have $C(OH_n,A)\approx (n/\log(n))^{1/2}$
  except in the very special case when $A=C \otimes_{\min} B$ with
 $C$ commutative and $B$ f.d.  in which case (assuming  $C$ infinite dimensional)
 $C(OH_n,A)\approx n^{1/2}$. Furthermore he shows
 that  $C(OH_n,A)\ge c (n/\log(n))^{1/2}$ for all $n>1$  (for some $c>0$) only if $A$ has the WEP.
 When $A$ is a von Neumann algebra   this yields a characterization of injectivity.
 While our argument for $F_n$ seems somewhat intertwined with Junge's for $OH_n$, we do not see
 either one as a consequence of the other.\\
 The reader will of course have noticed that
 the preceding proof shows that for   $M$ to be injective it
 suffices to have an estimate of the form
 $C(F_n,M) \ge c n^{1/2} (\log(n))^{-\b}$ for any fixed $\b>0$.
 Junge made a similar observation for $C(OH_n,A) \ge c n^{1/2} (\log(n))^{-\b}$  in his characterization of the WEP.
 \end{rem}

     \section{m.b. versus c.b. :   concrete examples}\label{con}
     
     In this section our goal will be to produce examples of derivations
     or unital homomorphisms that are c.b. but not m.b.
     Part of our original  motivation was to correct  a  certain erroneous
     statement  from \cite{[19]}.
     
             We will crucially use
  the o.s. structure of   the span of the free generators
   in $C^*_\lambda(\F_n)$ for the free groups $\F_n$ or $\F_\infty$.
Let $(g_j)_{1\le j\le n}$ denote the set of free generators
of   $\F_n$, or the first $n$ free generators of $\F_\infty$.
Let $G=\F_n$ or $G=\F_\infty$.
 We denote by $\lambda$ and $\rho$ the left and right regular representations of    $G$. 
  Let $M_\lambda=L(G)$  (i.e. $M_\lambda=C^*_\lambda(G)''$),  the von Neumann algebra generated by $\lambda$   with its   normalized trace  $\tau$.
Let $E_n \subset M_\lambda$ denote the linear span of
 $\{\lambda(g_j)\mid {1\le j\le n}\}$. 
 
  Haagerup   proved in \cite{Hainv}
   that the natural orthogonal projection $P: M_\lambda \to E_n$ 
   has $\|P\|_{cb} \le 2$, from which follows that
 \begin{equation}\label{e22}
   C(E_n,M_\lambda)\le 2.
   \end{equation}

           \begin{rem}  
 
      Let $H=\ell_2(G)$ ($G=\F_n$ or $G=\F_\infty$)  and let $(\d_t)_{t\in G} $ be its standard orthonormal basis.
     For any $g\in \{g_j\mid {1\le j\le n}\}$ let $H(g)\subset H$ denote the closed span of
     the family $\{\d_t\} $ for the (reduced) words $t$ with last letter equal to $g$.
    Since $ \rho(g) (1-P_{H(g)}) = P_{H(g^{-1})} \rho(g)$, one has a decomposition (cf. \cite[p. 896]{HP2})
      $$\rho(g)= \rho(g) P_{H(g)}  + P_{H(g^{-1})}\rho(g). $$
       For a family $(a_j)$ in $B(H)$,  we use the customary notation
     $$|(a_j)|_C=\|\sum a^*_ja_j\|^{1/2} \text{  and  } |(a_j)|_R=\|\sum a_ja^*_j\|^{1/2}.$$
      Let
     $$a_j=\rho(g_j) P_{H(g_j)}  \text{  and  } b_j=P_{H(g_j^{-1})}\rho(g_j) $$
     so that $\rho(g_j)=a_j+b_j$.
 Note that $|(a_j)|_C \le 1$ and $|(b_j)|_R \le 1$. Similarly by \cite{HP2} the same decomposition holds for  $\lambda$ as follows:
  \begin{equation}  \label{e17}
      \lambda(g_j)=\a_j +\b_j \text{  with  }
     |(\a_j)|_C \le 1\text{ and }  |(\b_j)|_R \le 1. \end{equation}
        
      \end{rem}
      
         \begin{rem} It is easy to check that
         \begin{equation}  \label{e41} C(E_n,B(H))\ge \sqrt {n}/2 .\end{equation}
         Indeed this boils down to checking that the ``constant of injectivity"
         of $E_n$ is at least $\sqrt {n}/2 $.
          It suffices to  check that if
         $Q: B(H) \to E_n$ is a linear projection, then $\|Q\|_{cb} \ge \sqrt {n}/2 $. 
         We have a decomposition 
         $\lambda(g_j)=\a_j+\b_j$ in $B(H)$  
         whence a decomposition 
         $\lambda(g_j)=Q\a_j+Q\b_j$ in $M_\lambda$.
         Therefore
         $\sqrt n=(\sum\| \lambda(g_j)\|_2)^{1/2}
         \le (\sum\| Q\a_j\|_2)^{1/2}
         +(\sum\| Q\b_j\|_2)^{1/2}
         \le |(Q\a_j)|_C +|(Q\b_j)|_R   
         \le \|Q\|_{cb}  |(\a_j)|_C +\|Q\|_{cb}|(\b_j)|_R\le
         2\|Q\|_{cb} $.\\
         The same argument shows that   $C(E_n,B)\ge \sqrt {n}/2 $ for any $B$ with WEP,
         since in that case the inclusion $B\to B^{**}$ factors through $B(H)$ and $w$ extends to $w^{**}$ on $B^{**}$. \\Note  that since $d_{cb}(E_n,F_n)\le 2$ (see Remark \ref{R3}) we could also derive from Theorem \ref{hp2}
         that $C(E_n,B(H))\ge (1+\sqrt {n})/4$.
          \end{rem}  
          
         \begin{thm}
         \label{T10} Let $c>0$  and $n\ge 1$. Let $G=\F_n$, $H=\ell_2(G)$ and
         let     $A=C^*_\lambda(G)\subset B(H)$.
Let $M$ be a von Neumann algebra such that
         $$\|\d\|_{mb} \le c \|\d\|_{cb} $$
         for any derivation $\d: A \otimes 1_M \to B(H)\bar\otimes M$.
         Then
         \begin{equation}  \label{e32} \sqrt n/2 \le 4 c C(E_n,M).\end{equation}
         Moreover, the same property holds for  $G=\F_\infty$.
        \end{thm}
        
         \begin{proof} We first claim that there is a derivation
         $D: A \otimes 1 \to B(H)\bar\otimes M_\lambda$ (here for simplicity $1$ denotes the unit
         in $M_\lambda$)
         with $\|D\|_{cb}\le 4$  and such that
         $D(\lambda(g_j) ) = P_e \otimes \lambda(g_j)$
         where $P_e\in B(H)$ is a element of unit norm (namely $P_e$ is the orthogonal projection onto $\CC\d_e$) and $g_j$ are the free generators of $G$.
         
         We set for any $x\in A$
         $$D(x\otimes 1)=[x\otimes 1, \Sigma]$$
         with
          $$\Sigma= \sum  - b_j \otimes \a_j +a_j \otimes \b_j .$$
         By a well known operator version of the Cauchy-Schwarz inequality (see e.g.  \cite[Lemma 2.4]{Ha1132})         we have
         $\|\sum b_j \otimes \a_j\|\le \|(b_j)\|_R \|(\a_j)\|_C\le 1$ and
         $\|\sum a_j \otimes \b_j\|\le \|(a_j)\|_C\|(\b_j)\|_R \le 1$ and hence $\|\Sigma\|\le 2$.
         Therefore $\|D\|_{cb} \le 4$.
         Since $\rho$ and $\lambda$ commute, $\rho(g_j)=a_j+b_j$    commutes with $x\in M_\lambda$ which means $[x,b_j]=-[x,a_j]$ and hence
         \begin{equation}  \label{e30} D(x\otimes 1) =[x\otimes 1, \Delta]\end{equation}
        where
       \begin{equation}  \label{e31} \Delta=   \sum\nl_1^n a_j \otimes \lambda(g_j)\in B(H) \otimes E_n  .\end{equation} 
  In particular it is important to observe that for all $x\in A$
  \begin{equation}  \label{e19}
  D(x\otimes 1) \in B(H) \otimes E_n  .
  \end{equation}        
    Let $a_g=\rho(g) P_{H(g)}$. 
     It is easy to check 
  that $[x, a_g]=0$ when $x=\lambda(g')$ with $g'\not=g$
     and $=P_e$  when $x=\lambda(g)$ so that
     $$D(\lambda(g_j)) = P_e \otimes \lambda(g_j) .$$
     This proves the claim.\\
     Consider a factorization of $Id_{E_n}$ through $M$ of the form
  $  E_n {\buildrel {w}\over \rightarrow} M {\buildrel {v}\over \rightarrow}  E_n $.
  
  Let $\d: A\otimes 1_M \to B(H)\bar\otimes M $ be the derivation defined by
  \begin{equation}  \label{e29}
  \d(x\otimes 1_M)=[ Id_{B(H)}\otimes w] D(x\otimes 1) =[x\otimes 1_M, [Id_{B(H)}\otimes w](\Delta) ].\end{equation}
  By our assumption
   $\|\d\|_{mb} \le c\|\d\|_{cb}\le c\|w\|_{cb} \|D\|_{cb}\le 4c\|w\|_{cb}.$ 
  
  By Corollary \ref{ptruc} our derivation
  $\d$ admits an extension $\hat \d: B(H) \otimes 1_M \to B(H)\bar \otimes M$ with
 $$\|\hat \d\|_{cb} \le \|\d\|_{mb} \le 4c \|w\|_{cb}.$$
 Then    $\hat\d(\lambda(g_j) \otimes 1_M)=\d(\lambda(g_j) \otimes 1_M)=P_e \otimes w(   \lambda(g_j))$, and hence $\hat\d(x \otimes 1_M) =P_e \otimes w(  x)$ for all $x\in E_n$.
 Let $f\in B(H)^*$ be of unit norm such that $f(P_e)=1$.
  We may view the map $f\otimes v$
  as one from ${B(H) \bar\otimes M }$ to ${E_n }$ with c.b. norm $\le \|v\|_{cb}$
  since obviously $\|f: B(H) \to \CC\|_{cb}=1$.
 Let $w': E_n \to B(H)$  and  $v': B(H) \to E_n$ be defined by $w'(x)= x\otimes 1_M$
 and let $v'(b) = [f  \otimes   v    ]\hat \d(b\otimes 1) $. 
 Then $v'w'(x) =vw(x)$ for all $x\in E_n$. 
 Thus $v'w' =Id_{E_n}$ is a factorization through $B(H)$ and hence if $c=\|  \d\|_{mb}  \|\d\|^{-1}_{cb}$
 $$C(E_n,B(H))\le \|v'\|_{cb} \|w'\|_{cb} \le \|v\|_{cb} \|\hat \d\|_{cb} 
  \le \|v\|_{cb} \|  \d\|_{mb} =  c \|v\|_{cb}  \|\d\|_{cb}  \le 4c \|v\|_{cb}\|w\|_{cb}.$$
 Taking the infimum over $v,w$ we obtain
 $$C(E_n,B(H))\le 4c C(E_n,M),$$
   and we conclude by \eqref{e41}.

 Lastly, assume now that $G=\F_\infty$.  
  Let $(g_j)$ denote the free generators of $\F_\infty$. 
 Then \eqref{e29}, \eqref{e30} and \eqref{e31}  can still be used to define the desired derivation, using for
 $(g_1,\cdots,g_n)$ the first free generators of $\F_\infty$.
 \end{proof}

      \begin{cor}\label{exa} 
      Let $G=\F_\infty$, $A=C^*_\lambda(G)$ and  $H=\ell_2(G)$.
      Assume that
      all c.b. derivations from $A\otimes 1_M$ to $B(H)\bar \otimes M $ are m.b.
      Then
       there is $c'>0$ such that $C(E_n,M)\ge c' \sqrt n$ for all $n\ge 1$.
  \end{cor}
      
    \begin{proof}
   By the closed graph theorem, if the (closed) subspaces 
   of derivations in $MB(A,B(H)\bar \otimes M)$ and
   $CB(A,B(H)\bar \otimes M)$ coincide, they must have equivalent norms, and then
    the corollary follows from \eqref{e32}  by the last assertion of  Theorem \ref{T10}.
   \end{proof}
   
   By \eqref{e22}  this implies:
  \begin{cor}\label{exa1} 
      Let $G=\F_\infty$, $A=C^*_\lambda(G)$ and  $M_\lambda=C^*_\lambda(G)''$.          There is a c.b. derivation  from $ A\otimes 1 $ to $ B(H)\bar \otimes M_\lambda$ that is not m.b.
      \end{cor}

 \begin{rem}By Remark \ref{R2}, if $Q(x\otimes 1)=q(x)\otimes 1$, we have
    $\| D Q\|_{mb}=\|D\|_{mb}$ for any map $D$ defined on $A\otimes 1$.
      Therefore we also obtain
       a similar example of a derivation from
      $\C\otimes 1\simeq\C$.\end{rem}

 Let $\pi: A \to M$ be a $*$-homomorphism and
  $\d: A \to M$ be a $\pi$-derivation (meaning $\d(xy)=\pi(x) \d(y)+\d(x)\pi(y)$ for all $x,y\in A$). 
 Then
      \begin{equation}  \label{e18}
      A \ni a \mapsto  u(a)=\left( \begin{matrix} \pi(a)  & \ \d(a)\\ 0& \  \pi(a)   \end{matrix}  \right)\end{equation}  
      is a homomorphism from $A$ to $M_2(M)$.

\begin{cor} Let $N=M_2(B(H)\bar \otimes M_\lambda)$.
There is a c.b. unital homomorphism $u: \C \to N$ that does not admit any c.b. lifting $\hat u: \C \to W$ when $W$ is any WEP
$C^*$-algebra of which  $N$ is a quotient.
\end{cor}
\begin{proof} Since $W$ is WEP we have
$\|\hat u\|_{mb} =\|\hat u\|_{cb}$ for any   $\hat u\in CB( \C , W)$.
Let $q: W \to N$ be a quotient map so that $u=q\hat u$.
Then $\|u\|_{mb} \le \|q\|_{mb} \|\hat u\|_{mb} \le \|\hat u\|_{cb}$.
When $u$ is associated to a derivation $\d: \C \to B(H)\bar \otimes M$ as in \eqref{e18}, 
we have clearly $\|\d\|_{mb} \le \|u\|_{mb}$, so
this gives us $\|\d\|_{mb} \le \|\hat u\|_{cb}$.
Taking Remark \ref{R2} into account, we can select a $u$ associated to a  c.b. $\d$
with $\|\d\|_{mb}=\infty$ and 
 the corollary follows.
\end{proof}

           \begin{rem} Actually the preceding argument shows that
           in the case $n=\infty$ there is no $c<\infty$ for which the homomorphism $u$
           is   locally $c$-liftable  in the sense of \cite[Def. 7.46]{154}.
           \end{rem}     
 
 \begin{rem}[Free gaussian analogue]
   Let $\cl F$ denote the full Fock space based
         on a Hilbert space with orthonormal basis $(e_j)$ ($1\le j\le n$).
         Consider in $B(\cl F)$ the left (resp. right) creation operator $\ell_j$ (resp. $r_j$)
         taking $x\in \cl F$ to $e_j \otimes x$ (resp. $  x\otimes e_j$). Let $\Omega$ denote the vacuum vector. Let $s_j=\ell_j+\ell_j^*$ and $t_j=r_j + r_j^*$.
         In Voiculescu's free probability theory \cite{VDN}, the families $(s_j)$ and $(t_j)$ are mutually commuting and each of them is a free gaussian (or semicircular) family with respect to the state
         $f_\Omega: x\mapsto \langle \Omega, x\Omega \rangle$.
         Let $\cl A$ denote the $C^*$-algebra generated by $(s_j)$ in $B(\cl F)$ and let $\cl M$
         denote the von Neumann algebra generated by $\cl A$. Note that $\cl M' =\cl A'$ is generated by $(t_j)$.
         When restricted to $\cl M$ or to $\cl M'  $, the state  $f_\Omega$ defines a normalized trace
         that we denote by $\tau$.
         Let $\cl E_n$ denote the linear span of $(s_j)_{1\le j\le n}$ viewed as an operator subspace of
         $\cl M$.
         It is  known that the orthogonal projection from $L_2(\tau)$ onto $\cl E_n$
         defines a projection $Q: \cl M\to \cl E_n$ with $\|Q\|_{cb} \le 2$.
         Let $P=Q_{|\cl A}:\cl A \to  \cl M  $.
          In complete analogy with
         Theorem \ref{T10} we can produce a derivation
         \def\D{{\cl D}}
         $\cl D: \cl A \otimes 1 \to B(\cl F) \bar \otimes \cl M$
          such that $x\mapsto [f_\Omega \otimes Q] \D (x \otimes 1)$ is
         equal to $P(x)$ for all $x\in \cl A$. 
         One simply sets
         $$\D(a\otimes 1)=\sum [a,r_j] \otimes s_j .$$
          As for Theorem \ref{T10} we can check that $\|\cl D\|_{cb} \le 4$ and $\sqrt n \le 2\|\D\|_{mb}$.
\end{rem}

                \begin{rem}\label{R3}
By Haagerup's by now classical results in \cite{Ha}
   the map $w: F_{n} \to E_n$ defined by 
   $w([e_{1j}\oplus e_{j1}])=\lambda(g_j)$ has $\|w\|_{cb} \le 2$ and the inverse map
   $w^{-1}$ admits an extension $v: M_\lambda \to  F_{n}$ with
   $\|v\|_{cb}\le 1$. In particular $d_{cb}(E_n,F_n)\le 2$, and also $d_{cb}(\cl E_n,F_n)\le 2$.  
   By \eqref{e128} this implies that for any o.s. $M$ we have $2^{-1} C(\cl E_n,M) \le C(  E_n,M) \le 2 C(\cl E_n,M)$
   and also:
    \begin{equation}  \label{e127} 
    2^{-1} C(  E_n,M) \le C(  F_n,M) \le 2 C( E_n,M).
   \end{equation}
   The proof that $\|w\|_{cb} \le 2$ can be derived easily from the decomposition \eqref{e17} : setting 
   $ w=w_1+w_2$ with $w_1([e_{1j}\oplus e_{j1}]) = a_j$ and $w_2([e_{1j}\oplus e_{j1}]) =b_j $
   we obtain c.c. maps $w_1,w_2$ from $F_n$ to $B(H)$ such that
   $ w=w_1+w_2$ on $F_n$. 
   As for $v: M_\lambda \to  F_{n}$, we use Remark \ref{homo}:
   denoting by $k:  M_\lambda \to L_2(M_\lambda, \tau)_{rc}$
   the inclusion in Lemma \ref{rcoh1} the o.s. $F_n$ can be identified completely isometrically with the operator subspace
   $F=k(E_n)\subset L_2(M_\lambda, \tau)_{rc}$, and the
   orthogonal projection $P_F$ on $L_2(M_\lambda, \tau)_{rc}$   has $\|P_F\|_{cb}\le 1$.
   Thus we can take for 
   $v$ the inclusion $k: M_\lambda \to L_2(M_\lambda, \tau)_{rc}$ 
   followed by   $P_F$.   
      In particular this implies
  \begin{equation}  \label{e27} C(F_{n}, M_\lambda)\le 2.\end{equation}

   \end{rem}
        
           \section{Similar versus strongly similar for homomorphisms into $N$}\label{svss}
 
         Nuclear $C^*$-algebras are characterized in \cite{toho} 
          by the property that any c.b. homomorphism
          with domain   the algebra (and range any von Neumann algebra) is strongly similar
          to a $*$-homomorphism.
         It is natural to wonder whether there is an analogous characterization
         of the von Neumann algebras $N$ such that any c.b. homomorphism
          with range  $N \subset B(H)$ (and domain any  $C^*$-algebra) is strongly similar
          (i.e. similar via an invertible in $N$)
          to a $*$-homomorphism. We will show 
          in Theorem \ref{cfin}
          that this holds for $N$ if and only if $N$ is the direct sum of a {\it finite} von Neumann algebra
          and an injective one.

          We start by a simple observation.
          
          \begin{pro}\label{pinj} Let  $N$ be an   injective von Neumann algebra.   
          \\ For any unital $C^*$-algebra $A$
          any unital c.b. homomorphism $u$  from $A$  to $N$  
          is strongly similar
          to a $*$-homomorphism, and  we have $\|u\|_{mb}=\|u\|_{cb}$.\\
           Moreover, any c.b. derivation $\d: A\to M$ relative to a $*$-homomorphism $\pi:A\to M$
          is inner and $\|\d\|_{mb}=\|\d\|_{cb}$.
          \end{pro}
          \begin{proof}  
           Since injective von Neumann algebras have the WEP, the equality $\|u\|_{mb}=\|u\|_{cb}$
           follows from (i) $\Rightarrow$ (ii) in Theorem
           \ref{wep}.
                      Strong similarity follows by Theorem \ref{s1}.\\
          Assume $M\subset B(H)$. Consider a c.b. derivation $\d:A \to M$ relative to $\pi$.
        Viewing it as  $B(H)$-valued we find $\|\d\|_{cb}=2\min\{\|T\| \mid T\in B(H) \ \d=\d^T\}$.
        Let $P: B(H) \to M$ be a conditional expectation projecting $B(H)$ onto $M$.
        Note $\d=\d^T$ implies $\d=\d^{P(T)}$, with  $P(T)\in M$  and $\|P(T)\|\le \|T\|$. Therefore
        by \eqref{chr}  $\|\d\|_{mb} \le \|\d\|_{cb}$ and the rest  is clear.
          \end{proof}
          \begin{rem} In the terminology of c.b. Hochschild cohomology presented in \cite{SS}
          the last assertion in Prop. \ref{pinj} means that
          $H^1_{cb}(\pi(A),M)=\{0\}$ for any $\pi$, or simply $H^1_{cb}(A,M)=\{0\}$ for any unital 
          $C^*$-subalgebra $A\subset M$.
          \end{rem}
     \begin{thm}\label{T55} 
     A von Neumann algebra
  $\M$  is injective if and only if
     all c.b. homomorphisms 
    with range $B(\ell_2) \bar\otimes \M$ are strongly similar to a $*$-homomorphism
    (or equivalently are m.b.).
    \end{thm}

           \begin{proof}   Let $H=\ell_2$. 
           Assume all  c.b. homomorphisms 
    with range $B(H) \bar\otimes \M$ are m.b. This assumption on $B(H) \bar\otimes \M$ obviously passes to $M_2(B(H) \bar\otimes \M)$.
     It follows (see  \eqref{e18})
     that all c.b. derivations with range in $B(H) \bar\otimes \M$ are m.b.  which implies
     by an elementary argument that the c.b. and m.b. norms must be equivalent
     on the space of derivations considered in Theorem \ref{T55}.
     Then  $\M$ is injective by Theorem \ref{t3.1}, Corollary \ref{exa} and \eqref{e127}. \\
     Conversely, if  $\M$ is injective
     so is $N=B(H) \bar\otimes \M$ 
     so the converse holds by Proposition \ref{pinj}.
   \end{proof} 
                    The preceding result leaves out the finite case. Fortunately,
           a classical theorem due to Vasilescu and Zsid\'o \cite{VZ} tells us that
           strong similarity is automatic when the c.b. (or even merely bounded) homomorphism takes values in a {\it finite} 
           von Neumann algebra. This allows us to conclude:
           
           \begin{thm}\label{cfin}  Let $M$ a von Neumann algebra. The following are equivalent:
       \begin{itemize}     \item[{\rm (i)} ] Any c.b. unital homomorphism into $M$ is strongly similar to a $*$-homomorphism.
           \item[{\rm (ii)}] Any c.b. unital homomorphism into $M$ is m.b.
        \item[{\rm (iii)}] $M$ can be decomposed as $M=M_1\oplus M_2$ with $M_1$ injective and $M_2$ finite.
        \item[{\rm (iv)}] Any c.b. derivation $\d: A\to M$ relative to a unital $*$-homomorphism $\pi:A\to M$
        (on any unital $C^*$-algebra $A$)
          is m.b.
            \item[{\rm (v)}] Any c.b. derivation $\d: A\to M$ relative to a $*$-homomorphism $\pi:A\to M$
             (on any unital $C^*$-algebra $A$)
          is inner.\\ Equivalently, $H^1_{cb}(A, M)=\{0\}$ for any unital $C^*$-subalgebra $A\subset M$.
              \end{itemize}
\end{thm}
 \begin{proof} The first two assertions are equivalent by Theorem \ref{s1}.
 Assume (i). By the structure theory of von Neumann algebras (see e.g. \cite[III.1.4.1]{bla} and  \cite[III.1.5.6]{bla}), we have a decomposition 
 of the form $M=M_1\oplus M_2$ with $M_2$ finite and $M_1$ properly infinite, which implies $M_1=B(\ell_2)\bar\otimes M_1$. 
 Clearly each of    $M_1, M_2$ inherits the  property (i).
 By Theorem \ref{T55} $M_1$ is injective, which gives us (iii).
 Assume (iii). Then $M_1$ satisfies (i) by Proposition \ref{pinj} and $M_2$ does by Vasilescu and Zsid\'o's
 result in \cite{VZ}. Then (i) follows. Thus (i), (ii) and (iii) are equivalent.\\
 We already know (iv) $\Leftrightarrow$ (v) by Theorem \ref{t5.6}. Clearly $M$ satisfies (iii) $\Leftrightarrow$ $M_2(M)$ satisfies (iii) $\Leftrightarrow$ $M_2(M)$ satisfies (ii) and the latter implies that $M$ satisfies  (iv) (see \eqref{e18}). This shows (iii) $\Rightarrow$ (iv). Assume (iv).
 Note that for any decomposition $M=M_1\oplus M_2$
 the inclusion $M_1\subset M$ is max-isometric, so that $M_1$ inherits (iv) from $M$.
  Let $M_1=B(\ell_2)\bar\otimes N$ be the properly infinite part  of $M$.
  By Corollary \ref{exa} and Theorem \ref{t3.1}
  $N$ and hence $M$ must be injective, which means that (iii) holds. Thus (iv) $\Leftrightarrow$ (iii).
 \end{proof}
          
\begin{rem} It is erroneously claimed in \cite{fang} (using \cite[Lemma 2.3]{[19]}) that all QWEP von Neumann algebras satisfy
(i) in Theorem \ref{cfin}.
\end{rem}

\end{document}